\documentclass{article}
\usepackage[text={170mm,254mm}]{geometry}
\usepackage{graphicx}
\usepackage[T1]{fontenc}
\usepackage[utf8]{inputenc}

\usepackage{amssymb, amsthm, amsfonts, amsmath, enumitem, mathtools}
\usepackage{booktabs, array, diagbox, xcolor, multicol, rotating, stmaryrd, multirow}

\usepackage{algorithmicx}
\usepackage{algpseudocode}

\usepackage[hyphens]{url}
\usepackage[colorlinks=true,citecolor=black,linkcolor=black,urlcolor=blue]{hyperref}
\usepackage{cite}

\usepackage{subcaption}
\usepackage{float}

\newcommand{\arxiv}[2]{\href{https://arxiv.org/abs/#1}{\texttt{arXiv:#1}} \texttt{[#2]}}
\newcommand{\doi}[1]{\url{https://doi.org/#1}}

\theoremstyle{plain}
\newtheorem{theorem}{Theorem}[section]
\newtheorem{lemma}[theorem]{Lemma}
\newtheorem{conjecture}[theorem]{Conjecture}
\newtheorem{proposition}[theorem]{Proposition}
\newtheorem{corollary}[theorem]{Corollary}
\theoremstyle{definition}

\newtheorem{example}[theorem]{Example}

\usepackage{listings}
\definecolor{mauve}{rgb}{0.58,0,0.82}
\definecolor{dkgreen}{rgb}{0,0.6,0}
\lstdefinestyle{pitonche} {
    language = Python,
    basicstyle = footnotesizettfamily,
    showspaces = false,
    showstringspaces = false,
    breakautoindent = true,
    flexiblecolumns = true,
    keepspaces = true,
    stepnumber = 1,
    xleftmargin = 0pt
}
\usepackage{tikz}

\let\le=\leqslant
\let\ge=\geqslant

\usepackage{authblk}

\title{Computation of small reflective and dihedral Ramsey numbers}

\author[1,2]{Ivan Damnjanović\thanks{Corresponding author. Email:\ \texttt{ivan.damnjanovic@elfak.ni.ac.rs} (I.\ Damnjanović).}}
\author[1]{Irena Đorđević}

\affil[1]{Faculty of Electronic Engineering, University of Niš, Aleksandra Medvedeva 4, Niš, 18104, Serbia}
\affil[2]{Faculty of Mathematics, Natural Sciences and Information Technologies, University of Primorska,\linebreak Glagoljaška 8, Koper, 6000, Slovenia}

\date{}

\begin{document}

\maketitle

\begin{abstract}
Throughout, all graphs are simple, finite and have vertex sets of the form
$\{ 0, 1, 2, \ldots, n - 1 \}$ for some $n \in \mathbb{N}$. For graphs $G$ and $H$, and a permutation group $\Gamma$ on the vertex set of $H$, we say that $H$ is \emph{$\Gamma$-embeddable} in $G$ if there exists a graph homomorphism from $H$ to $G$ of the form $\psi \circ \varphi$, where $\varphi \in \Gamma$ and $\psi$ is an increasing injection. Recently, standard and ordered Ramsey numbers of graphs were unified through the introduction of permutational Ramsey numbers, defined as follows. For graphs $H_1, H_2, \ldots, H_k$ and permutation groups $\Gamma_1, \Gamma_2, \ldots, \Gamma_k$ on their respective vertex sets, the \emph{permutational Ramsey number} $R(H_1^{\Gamma_1}, H_2^{\Gamma_2}, \ldots, H_k^{\Gamma_k})$ is the minimum $n \in \mathbb{N}$ such that for every $k$-edge-coloring of a complete graph on $n$ vertices, there exists some $j \in \{1, 2, \ldots, k\}$ for which $H_j$ is $\Gamma_j$-embeddable in the spanning subgraph of the complete graph comprising the edges of color $j$.

Here, we consider reflective (resp.\ dihedral) Ramsey numbers, which are a specific class of permutational Ramsey numbers in which each group $\Gamma_j$ is the reflection group (resp.\ dihedral group) on the naturally ordered vertex set of $H_j$. Focusing on the two-color case, we apply the SAT-based approach originally proposed by Poljak for ordered Ramsey numbers and recently extended to cyclic Ramsey numbers. We utilize the Kissat SAT solver to obtain exact values and lower bounds for small reflective and dihedral Ramsey numbers whose two arguments belong to the following graph classes: monotone and alternating paths, monotone cycles, start-central stars, complete graphs and nested matchings. We also derive several general results and formulate conjectures based on the computational findings. 
\end{abstract}



\section{Introduction}

We consider all graphs to be undirected, simple and finite, with vertex sets of the form $\{ 0, 1, 2, \ldots, n - 1 \}$ for some $n \in \mathbb{N}$. For a graph~$G$, we denote its vertex and edge sets by $V(G)$ and $E(G)$, respectively, and write $|G|$ for its order. We write $K_n$ for the complete graph on $n$ vertices. A $k$-edge-coloring of a graph $G$ is a mapping from $E(G)$ to $\{ 1, 2, \ldots, k \}$.

Ramsey theory originates from a theorem of Ramsey \cite{Ramsey1930}, which asserts that for any graphs $H_1, H_2, \ldots, H_k$, every $k$-edge-coloring of $K_n$ contains a monochromatic copy of $H_j$ in color $j$, for some $j \in \{ 1, 2, \ldots, k \}$, provided that $n \in \mathbb{N}$ is sufficiently large. The minimum such $n$ is called the \emph{Ramsey number} $R(H_1, H_2, \ldots, H_k)$. Although the definition of Ramsey numbers is straightforward, determining their exact values is often remarkably difficult. Indeed, even the diagonal Ramsey number $R(K_5, K_5)$ remains unknown \cite{Radziszowski}. Ramsey-type phenomena have been investigated in numerous settings over the past century, leading to a rich literature on a variety of combinatorial structures, including hypergraphs, ordered graphs, geometric configurations, and sequences; see \cite{ConFoxSu2015, ErSze1935, FraPaReiRo2018, GraRothSpe1990, LiLin2022, Morris2026, NeRo1990, Nguyen2014, Radziszowski} and the references therein.

Among the many directions in modern Ramsey theory, ordered Ramsey numbers have received considerable attention in recent years. Given graphs $H_1, H_2, \ldots, H_k$, the \emph{ordered Ramsey number} $R_\mathrm{ord}(H_1, H_2, \ldots, H_k)$ is defined as the minimum $n \in \mathbb{N}$ such that for every $k$-edge-coloring of $K_n$, there exists an increasing injective homomorphism from $H_j$ to the spanning subgraph of $K_n$ comprising the edges of color $j$, for some $j \in \{ 1, 2, \ldots, k \}$. The systematic study of ordered Ramsey numbers was initiated independently by Balko, Cibulka, Král and Kynčl \cite{BalCiKralKyn2015, BalCiKralKyn2020}, and Conlon, Fox, Lee and Sudakov \cite{ConFoxLeeSu2017}. For a selection of recent results on the asymptotic properties of ordered Ramsey numbers, see the survey \cite{Balko2025} by Balko and the references therein.

Recently, standard and ordered Ramsey numbers were unified within a group-theoretic framework through the introduction of permutational Ramsey numbers \cite{BaDamSteSto2026}. For convenience, we first introduce the following terminology. Let $G$ and $H$ be graphs, and let $\Gamma$ be a permutation group on $V(H)$. We say that $H$ is \emph{$\Gamma$-embeddable} in $G$ if there exists a graph homomorphism from $H$ to $G$ of the form $\psi \circ \varphi$, where $\varphi \in \Gamma$ and $\psi \colon V(H) \to V(G)$ is an increasing injection. Using this terminology, the definition of permutational Ramsey numbers from \cite{BaDamSteSto2026} can be stated as follows. For graphs $H_1, H_2, \ldots, H_k$ and permutation groups $\Gamma_1, \Gamma_2, \ldots, \Gamma_k$ on their respective vertex sets, the \emph{permutational Ramsey number} $R(H_1^{\Gamma_1}, H_2^{\Gamma_2}, \ldots, H_k^{\Gamma_k})$ is the minimum $n \in \mathbb{N}$ such that for every $k$-edge-coloring of $K_n$, there exists some $j \in \{1, 2, \ldots, k\}$ for which $H_j$ is $\Gamma_j$-embeddable in the spanning subgraph of $K_n$ comprising the edges of color $j$. The well-definedness of permutational Ramsey numbers follows directly from Ramsey's theorem \cite{Ramsey1930}.

Different choices of permutation groups give rise to several previously studied classes of Ramsey numbers. In particular, standard Ramsey numbers are obtained by taking each $\Gamma_j$ to be the symmetric group on $V(H_j)$, whereas ordered Ramsey numbers are obtained by taking each $\Gamma_j$ to be the trivial permutation group on $V(H_j)$. The same framework also encompasses cyclic Ramsey numbers $R_\mathrm{cyc}(H_1, H_2, \ldots, H_k)$, introduced and investigated in \cite{BaDamSteSto2026}, which arise by taking each $\Gamma_j$ to be the cyclic group generated by the cyclic shift permutation with respect to the natural vertex order, i.e., $\Gamma_j = \left\langle \begin{pmatrix}\begin{smallmatrix}
    0 & 1 & 2 & \cdots & |H_j| - 2 & |H_j| - 1\\
    1 & 2 & 3 & \cdots & |H_j| - 1 & 0
\end{smallmatrix}\end{pmatrix} \right\rangle$ for each $j \in \{ 1, 2, \ldots, k \}$.

Motivated by the results of \cite{BaDamSteSto2026} on the computation of small ordered and cyclic Ramsey numbers, we investigate two other natural classes of permutational Ramsey numbers proposed in the same paper:
\begin{enumerate}[label=\textbf{(\arabic*)}]
    \item the reflective Ramsey numbers $R_\mathrm{ref}(H_1, H_2, \ldots, H_k)$, which arise by taking each $\Gamma_j$ to be the group generated by the reflection with respect to the natural vertex order, i.e., $\Gamma_j = \left\langle \begin{pmatrix}\begin{smallmatrix}
    0 & 1 & 2 & \cdots & |H_j| - 1\\
    |H_j|-1 & |H_j|-2 & |H_j|-3 & \cdots & 0
\end{smallmatrix}\end{pmatrix} \right\rangle$ for each $j \in \{ 1, 2, \ldots, k \}$; and
    \item the dihedral Ramsey numbers $R_\mathrm{dih}(H_1, H_2, \ldots, H_k)$, which arise by taking each $\Gamma_j$ to be the group generated by the cyclic shift permutation and the reflection with respect to the natural vertex order, i.e.,
    \[
        \Gamma_j = \left\langle \begin{pmatrix}\begin{smallmatrix}
    0 & 1 & 2 & \cdots & |H_j| - 2 & |H_j| - 1\\
    1 & 2 & 3 & \cdots & |H_j| - 1 & 0
\end{smallmatrix}\end{pmatrix}, \begin{pmatrix}\begin{smallmatrix}
    0 & 1 & 2 & \cdots & |H_j| - 1\\
    |H_j|-1 & |H_j|-2 & |H_j|-3 & \cdots & 0
\end{smallmatrix}\end{pmatrix} \right\rangle
    \]
    for each $j \in \{ 1, 2, \ldots, k \}$. 
\end{enumerate}
It is natural to study reflective (resp.\ dihedral) Ramsey numbers because they correspond to graphs whose vertices are linearly (resp.\ cyclically) ordered, but the direction of the ordering is ignored, unlike in the ordered (resp.\ cyclic) Ramsey numbers.

We focus on the two-color case, with the goal of computing small reflective and dihedral Ramsey numbers whose arguments belong to the following graph classes: monotone and alternating paths, monotone cycles, start-central stars, complete graphs and nested matchings. Formal definitions of these graph classes are given in Section~\ref{sc:prel}. The Ramsey numbers are computed by formulating Boolean satisfiability (SAT) problems that encode the search for a $2$-edge-coloring of $K_n$ containing no forbidden embeddings, following the approach originally proposed by Poljak \cite{Poljak2020} for ordered Ramsey numbers and later extended to cyclic Ramsey numbers \cite{BaDamSteSto2026}. Due to its strong performance in the main track of the SAT Competition 2024, we use the Kissat SAT solver \cite{Kissat, KissatRepo} to handle all the resulting SAT instances. We automate the computation of small reflective and dihedral Ramsey numbers by using and extending the source code available at \cite{BaDamSteSto2026Repo}, originally developed for computing ordered and cyclic Ramsey numbers. We also establish several general results and formulate conjectures motivated by the computational findings.

The remainder of this paper is organized as follows. Section~\ref{sc:prel} introduces the necessary terminology, establishes some basic properties of permutational Ramsey numbers and formally defines the graph classes considered throughout the paper. Section~\ref{sc:meth} describes the SAT-based approach and the software used to compute small reflective and dihedral Ramsey numbers. Section~\ref{sc:results} presents the computational results, establishes several general results, formulates conjectures motivated by the computational findings and compares the computed Ramsey numbers with the previously known ordered and cyclic Ramsey numbers. Finally, Section~\ref{sc:conc} concludes the paper with a brief summary and directions for future research. The source code used to obtain the computational results is available at \cite{GitHub}, together with supplementary data, including the computed Ramsey numbers, the Kissat output files and all the $2$-edge-colorings yielding the lower bounds.

\section{Preliminaries}\label{sc:prel}

We begin with the following notion, originally introduced in \cite{BaDamSteSto2026}. Let $G_1$ and $G_2$ be graphs of the same order, and let $\Gamma$ be a permutation group on $V(G_1) = V(G_2)$. We say that $G_1$ is \emph{$\Gamma$-isomorphic} to $G_2$, and write $G_1 \cong_\Gamma G_2$, if there exists an isomorphism $\varphi$ from $G_1$ to $G_2$ such that $\varphi \in \Gamma$. The relation $\cong_\Gamma$ is clearly an equivalence relation. Moreover, graphs belonging to the same equivalence class are interchangeable as arguments of a permutational Ramsey number whenever the associated permutation group is $\Gamma$.

\begin{proposition}\label{graph_invar}
    Let $H_1, H_2, \ldots, H_k$ be graphs and let $\Gamma_1, \Gamma_2, \ldots, \Gamma_k$ be permutation groups on their respective vertex sets. Suppose that $F_1, F_2, \ldots, F_k$ are graphs such that $F_j \cong_{\Gamma_j} H_j$ for each $j \in \{ 1, 2, \ldots, k \}$. Then $R(F_1^{\Gamma_1}, F_2^{\Gamma_2}, \ldots, F_k^{\Gamma_k}) = R(H_1^{\Gamma_1}, H_2^{\Gamma_2}, \ldots, H_k^{\Gamma_k})$.
\end{proposition}
\begin{proof}
    Let $n \in \mathbb{N}$ and consider a $k$-edge-coloring of $K_n$. For each $j \in \{ 1, 2, \ldots, k \}$, let $G_j$ denote the spanning subgraph of $K_n$ comprising the edges of color $j$. Suppose that there exists $j \in \{ 1, 2, \ldots, k \}$ such that $H_j$ is $\Gamma_j$-embeddable in $G_j$. By symmetry, it suffices to show that $F_j$ is also $\Gamma_j$-embeddable in $G_j$. By the definition of $\Gamma_j$-embeddability, there exists a homomorphism from $H_j$ to $G_j$ of the form $\psi \circ \varphi$, where $\varphi \in \Gamma_j$ and $\psi \colon V(H_j) \to V(G_j)$ is an increasing injection. Since $F_j \cong_{\Gamma_j} H_j$, there exists an isomorphism $\chi \in \Gamma_j$ from $F_j$ to $H_j$. Therefore, $\psi \circ (\varphi \circ \chi)$ is a homomorphism from $F_j$ to $G_j$, with $\varphi \circ \chi \in \Gamma_j$ and $\psi$ an increasing injection. Hence, $F_j$ is $\Gamma_j$-embeddable in $G_j$.
\end{proof}

Complementing Proposition~\ref{graph_invar}, which treats changes of the arguments, we now consider changes in the associated permutation groups. The following observation establishes a monotonicity property with respect to these groups.

\begin{proposition}\label{sandwich_prop}
    Let $H_1, H_2, \ldots, H_k$ be graphs, and let $\Gamma_1, \Gamma_2, \ldots, \Gamma_k$ and $\Lambda_1, \Lambda_2, \ldots, \Lambda_k$ be permutation groups on their respective vertex sets. Suppose that $\Gamma_j \le \Lambda_j$ for each $j \in \{ 1, 2, \ldots, k \}$. Then $R(H_1^{\Lambda_1}, H_2^{\Lambda_2}, \ldots, H_k^{\Lambda_k}) \le R(H_1^{\Gamma_1}, H_2^{\Gamma_2}, \ldots, H_k^{\Gamma_k})$.
\end{proposition}

\begin{corollary}\label{sandwich_cor}
    Let $H_1, H_2, \ldots, H_k$ be graphs and let $\Gamma_1, \Gamma_2, \ldots, \Gamma_k$ be permutation groups on their respective vertex sets. Then $R(H_1, H_2, \ldots, H_k) \le R(H_1^{\Gamma_1}, H_2^{\Gamma_2}, \ldots, H_k^{\Gamma_k}) \le R_\mathrm{ord}(H_1, H_2, \ldots, H_k)$.
\end{corollary}

Given graphs $G$ and $H$, and a permutation group $\Gamma$ on $V(H)$, it is straightforward to observe that the following two statements are equivalent.
\begin{enumerate}[label=\textbf{(\arabic*)}]
    \item $H$ is $\Gamma$-embeddable in $G$.
    \item There exists a graph $F \cong_\Gamma H$ with an increasing injective homomorphism from $F$ to $G$.
\end{enumerate}

\noindent
This yields the following result.

\begin{proposition}\label{equiv_class}
    Let $H_1, H_2, \ldots, H_k$ be graphs, and let $\Gamma_1, \Gamma_2, \ldots, \Gamma_k$ and $\Lambda_1, \Lambda_2, \ldots, \Lambda_k$ be permutation groups on their respective vertex sets. Suppose that, for each $j \in \{ 1, 2, \ldots, k \}$, $H_j$ lies in the same equivalence class under both $\cong_{\Gamma_j}$ and $\cong_{\Lambda_j}$. Then $R(H_1^{\Lambda_1}, H_2^{\Lambda_2}, \ldots, H_k^{\Lambda_k}) = R(H_1^{\Gamma_1}, H_2^{\Gamma_2}, \ldots, H_k^{\Gamma_k})$.
\end{proposition}

A graph $H$ is said to have \emph{reflection symmetry} if the mapping
$\varphi \colon V(H) \to V(H)$ defined by $\varphi(x) = |H| - 1 - x$ is an automorphism of $H$. The following corollary of Proposition \ref{equiv_class} is immediate for graphs with reflection symmetry.

\begin{corollary}\label{ref_sym_cor}
    For any graphs $H_1, H_2, \ldots, H_k$ with reflection symmetry, we have $R_\mathrm{ref}(H_1, H_2, \ldots, H_k) = R_\mathrm{ord}(H_1, H_2, \ldots, H_k)$ and $R_\mathrm{dih}(H_1, H_2, \ldots, H_k) = R_\mathrm{cyc}(H_1, H_2, \ldots, H_k)$.
\end{corollary}

\noindent
The following corollary is another consequence of Proposition \ref{equiv_class}.

\begin{corollary}\label{ref_sym_ssc_cor}
    Let $H_1, H_2, \ldots, H_k$ be graphs such that each graph is a start-central star or has reflection symmetry. Then $R_\mathrm{dih}(H_1, H_2, \ldots, H_k) = R_\mathrm{cyc}(H_1, H_2, \ldots, H_k)$.
\end{corollary}

Now, we formally define each of the graph classes relevant to our work. For any $n \in \mathbb{N}$, the \emph{monotone path} $P_n^\mathrm{mon}$ is the graph of order $n$ in which any two vertices are adjacent if and only if they are consecutive integers, while the \emph{alternating path} $P_n^\mathrm{alt}$ is the path graph of order~$n$ with the underlying path $(0, n - 1, 1, n - 2, 2, n - 3, \linebreak \ldots, \lfloor \frac{n}{2} \rfloor)$; see Figure \ref{alt_fig}. For any $n \ge 3$, the \emph{monotone cycle} $C_n^\mathrm{mon}$ is the graph obtained from $P_n^\mathrm{mon}$ by adding the edge $\{ 0, n - 1 \}$. Also, for convenience, we let $C_2^\mathrm{mon} \coloneqq K_2$. Furthermore, for any $n \in \mathbb{N}$, the \emph{start-central star} $S_n^\mathrm{sc}$ is the star graph of order $n$ in which vertex $0$ is adjacent to all the other vertices. Finally, for any even $n \ge 2$, the \emph{nested matching} $M_n^\mathrm{nest}$ is the $1$-regular graph of order $n$ in which each vertex $v$ is adjacent only to $n - 1 - v$; see Figure \ref{nm_fig}.

The choice of graph classes of interest is motivated by the computational results of \cite{BaDamSteSto2026}. The same graph classes are considered, with the exception of reverse alternating paths, which arise by reflecting alternating paths and are therefore equivalent to them in the context of reflective and dihedral Ramsey numbers. Furthermore, by Corollaries~\ref{ref_sym_cor} and \ref{ref_sym_ssc_cor}, it suffices to consider reflective Ramsey numbers in which at least one argument lacks reflection symmetry, and dihedral Ramsey numbers in which at least one argument is an alternating path. Otherwise, the problem reduces to computing an ordered or a cyclic Ramsey number, respectively, which was already investigated in \cite{BaDamSteSto2026}. Thus, we are interested in computing the reflective Ramsey numbers of alternating paths and start-central stars versus all the considered graph classes, together with the dihedral Ramsey numbers of alternating paths versus all the considered graph classes.

\begin{figure}[t]
\centering
\subcaptionbox{The alternating path $P_5^\mathrm{alt}$.}[0.47\textwidth]
{
    \centering
    \begin{tikzpicture}
        \tikzstyle{vertex}=[draw,circle,font=\scriptsize,minimum size=4pt,inner sep=1pt,fill=black]
        \tikzstyle{edge}=[draw,thick]

        \foreach \i in {0,1,2,3,4} {
            \node[vertex] (v\i) at ({1.25*\i}, 0) {};
            \node[below=4pt] at (v\i) {$\i$};
        }

        \path[edge, bend left=60] (v0) to (v4);
        \path[edge, bend right=60] (v4) to (v1);
        \path[edge, bend left=60] (v1) to (v3);
        \path[edge, bend right=60] (v3) to (v2);
    \end{tikzpicture}
}
\qquad
\subcaptionbox{The alternating path $P_6^\mathrm{alt}$.}[0.47\textwidth]
{
    \centering
    \begin{tikzpicture}
        \tikzstyle{vertex}=[draw,circle,font=\scriptsize,minimum size=4pt,inner sep=1pt,fill=black]
        \tikzstyle{edge}=[draw,thick]

        \foreach \i in {0,1,2,3,4,5} {
            \node[vertex] (v\i) at ({1.25*\i}, 0) {};
            \node[below=4pt] at (v\i) {$\i$};
        }

        \path[edge, bend left=60] (v0) to (v5);
        \path[edge, bend right=60] (v5) to (v1);
        \path[edge, bend left=60] (v1) to (v4);
        \path[edge, bend right=60] (v4) to (v2);
        \path[edge, bend left=60] (v2) to (v3);
    \end{tikzpicture}
}
\caption{The alternating paths of orders five and six. Source: \cite[Figure~1]{BaDamSteSto2026}.}
\label{alt_fig}
\end{figure}
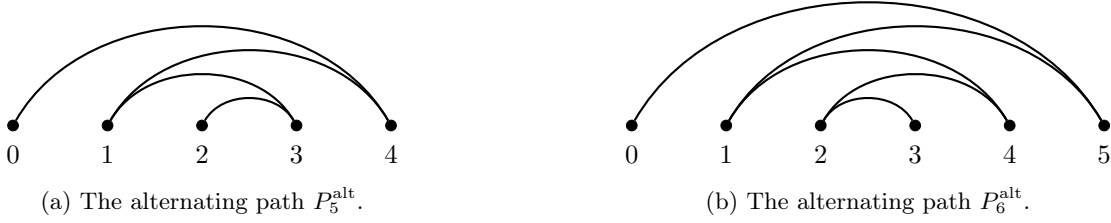

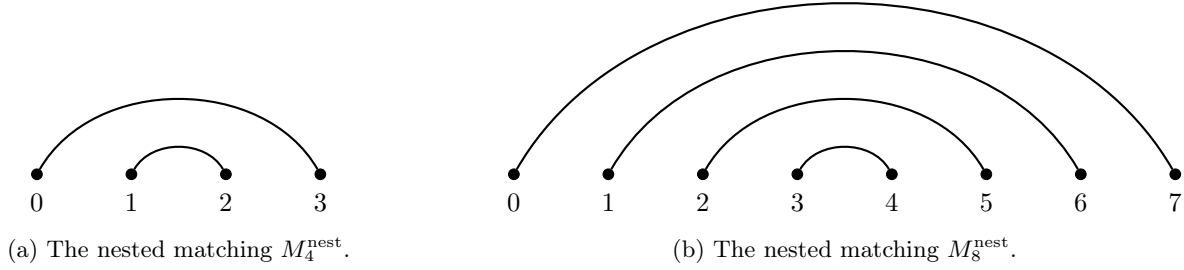
\begin{figure}[t]
\centering
\subcaptionbox{The nested matching $M_4^\mathrm{nest}$.}[0.32\textwidth]
{
    \centering
    \begin{tikzpicture}
        \tikzstyle{vertex}=[draw,circle,font=\scriptsize,minimum size=4pt,inner sep=1pt,fill=black]
        \tikzstyle{edge}=[draw,thick]

        \foreach \i in {0,1,2,3} {
            \node[vertex] (v\i) at ({1.25*\i}, 0) {};
            \node[below=4pt] at (v\i) {$\i$};
        }

        \path[edge, bend left=60] (v0) to (v3);
        \path[edge, bend left=60] (v1) to (v2);
    \end{tikzpicture}
}
\qquad
\subcaptionbox{The nested matching $M_8^\mathrm{nest}$.}[0.62\textwidth]
{
    \centering
    \begin{tikzpicture}
        \tikzstyle{vertex}=[draw,circle,font=\scriptsize,minimum size=4pt,inner sep=1pt,fill=black]
        \tikzstyle{edge}=[draw,thick]

        \foreach \i in {0,1,2,3,4,5,6,7} {
            \node[vertex] (v\i) at ({1.25*\i}, 0) {};
            \node[below=4pt] at (v\i) {$\i$};
        }

        \path[edge, bend left=60] (v0) to (v7);
        \path[edge, bend left=60] (v1) to (v6);
        \path[edge, bend left=60] (v2) to (v5);
        \path[edge, bend left=60] (v3) to (v4);
    \end{tikzpicture}
}
\caption{The nested matchings of orders four and eight. Source: \cite[Figure~3]{BaDamSteSto2026}.}
\label{nm_fig}
\end{figure}

We conclude this section with the following lemma.

\begin{lemma}\label{cool_lemma}
    Let $H_1, H_2, \ldots, H_k$ and $F$ be graphs such that there exists an increasing injective homomorphism from $F$ to $H_1$. Then $R_\mathrm{type}(F, H_2, \ldots, H_k) \le R_\mathrm{type}(H_1, H_2, \ldots, H_k)$ for each $\mathrm{type} \in \{ \mathrm{ord}, \mathrm{ref}, \mathrm{cyc}, \mathrm{dih} \}$.
\end{lemma}
\begin{proof}
    Let $n \in \mathbb{N}$ and consider a $k$-edge-coloring of $K_n$. Let $G$ denote the spanning subgraph of $K_n$ comprising the edges of color $1$. Furthermore, let $\Gamma$ and $\Lambda$ denote the permutation groups associated with $H_1$ and $F$, respectively, in the permutational Ramsey numbers $R_\mathrm{type}(H_1, H_2, \ldots, H_k)$ and $R_\mathrm{type}(F, H_2, \ldots, H_k)$. To complete the proof, it suffices to show that if $H_1$ is $\Gamma$-embeddable in $G$, then $F$ is $\Lambda$-embeddable in $G$.
    
    Let $\chi$ be an increasing injective homomorphism from $F$ to $H_1$, and suppose that there exists a homomorphism from $H_1$ to $G$ of the form $\psi_1 \circ \varphi_1$, where $\varphi_1 \in \Gamma$ and $\psi_1 \colon V(H_1) \to V(G)$ is an increasing injection. Then $f \coloneqq \psi_1 \circ \varphi_1 \circ \chi$ is an injective homomorphism from $F$ to $G$. We carry out the proof by showing that $f$ can be written as $f = \psi_2 \circ \varphi_2$ for some $\varphi_2 \in \Lambda$ and increasing injection $\psi_2 \colon V(F) \to V(G)$.

    If $\mathrm{type} = \mathrm{ord}$, then $\varphi_1$ is the identity permutation, so $f$ is increasing, which directly gives the desired result. On the other hand, if $\mathrm{type} = \mathrm{ref}$, then $\varphi_1$ is either the identity permutation, in which case we are done, or $\varphi_1$ is the reflection with respect to the natural vertex order. In the latter case, we have $\varphi_1(x) = |H_1| - 1 - x$ for each $x \in V(H_1)$, which implies that $f$ is decreasing. By defining $\varphi_2 \in \Lambda$ via $\varphi_2(x) = |F| - 1 - x$ for each $x \in V(F)$, it follows that $\psi_2 \coloneqq f \circ \varphi_2$ is an increasing injection from $V(F)$ to $V(G)$. Therefore, $f = \psi_2 \circ \varphi_2^{-1} = \psi_2 \circ \varphi_2$ is the desired representation of $f$.
    
    Next, suppose that $\mathrm{type} = \mathrm{cyc}$. Here, we have $\varphi_1 = g^p$ for $g = \begin{pmatrix}\begin{smallmatrix}
    0 & 1 & 2 & \cdots & |H_1| - 2 & |H_1| - 1\\
    1 & 2 & 3 & \cdots & |H_1| - 1 & 0
    \end{smallmatrix}\end{pmatrix}$ and some $p \in \mathbb{N}$. First, we show that for any increasing injection $h_1 \colon V(F) \to V(H_1)$, the map $g \circ h_1$ can be written as $g \circ h_1 = h_2 \circ r$ for some $r \in \Lambda$ and increasing injection $h_2 \colon V(F) \to V(H_1)$. If $h_1(|F| - 1) \neq |H_1| - 1$, then $g \circ h_1$ is increasing, so we are done. Otherwise, $(g \circ h_1)(|F| - 1) = 0$ and $1 \le (g \circ h_1)(0) \le (g \circ h_1)(1) \le \cdots \le (g \circ h_1)(|F| - 2)$. By letting $r \coloneqq \begin{pmatrix}\begin{smallmatrix}
    0 & 1 & 2 & \cdots & |F| - 2 & |F| - 1\\
    1 & 2 & 3 & \cdots & |F| - 1 & 0
    \end{smallmatrix}\end{pmatrix}$, it follows that $h_2 \coloneqq (g \circ h_1) \circ r^{-1}$ is an increasing injection, so $g \circ h_1 = h_2 \circ r$ is the desired representation of $g \circ h_1$. Now, by a simple induction on $p$, we conclude that for any increasing injection $h_1 \colon V(F) \to V(H_1)$, the map $g^p \circ h_1$ can also be written as $g^p \circ h_1 = h_2 \circ r$ for some $r \in \Lambda$ and increasing injection $h_2 \colon V(F) \to V(H_1)$. By applying this observation to $g^p \circ \chi$, we obtain $f = \psi_1 \circ g^p \circ \chi = \psi_1 \circ h_2 \circ r$, so the original desired representation follows by letting $\psi_2 \coloneqq \psi_1 \circ h_2$ and $\varphi_2 \coloneqq r$.

    Finally, suppose that $\mathrm{type} = \mathrm{dih}$, and let $g$ be defined as above. In this situation, either $\varphi_1 = g^p$ for some $p \in \mathbb{N}$, or $\varphi_1 = t \circ g^p$ for $t = \begin{pmatrix}\begin{smallmatrix}
    0 & 1 & 2 & \cdots & |H_1| - 1\\
    |H_1|-1 & |H_1|-2 & |H_1|-3 & \cdots & 0
    \end{smallmatrix}\end{pmatrix}$ and some $p \in \mathbb{N}$. In the former case, the result is proved in the same manner as in the case $\mathrm{type} = \mathrm{cyc}$. In the latter case, we have $f = \psi_1 \circ t \circ g^p \circ \chi$, so by the results from the case $\mathrm{type} = \mathrm{cyc}$, we conclude that $f$ can be written as $f = \psi_1 \circ t \circ h_2 \circ r$, where $r = z^q$ for $z = \begin{pmatrix}\begin{smallmatrix}
    0 & 1 & 2 & \cdots & |F| - 2 & |F| - 1\\
    1 & 2 & 3 & \cdots & |F| - 1 & 0
    \end{smallmatrix}\end{pmatrix}$ and some $q \in \mathbb{N}$, and $h_2 \colon V(F) \to V(H_1)$ is an increasing injection. Now, define $w \in \Lambda$ via $w(x) = |F| - 1 - x$ for each $x \in V(F)$, and observe that $\psi_2 \coloneqq \psi_1 \circ t \circ h_2 \circ w$ is an increasing injection from $F$ to $G$. Therefore, $\psi_1 \circ t \circ h_2 = \psi_2 \circ w^{-1} = \psi_2 \circ w$, which yields $f = \psi_2 \circ w \circ r$. The desired representation now follows by letting $\varphi_2 \coloneqq w \circ r$.
\end{proof}

\noindent
Our software for computing small reflective and dihedral Ramsey numbers relies on the following corollaries of Lemma~\ref{cool_lemma}.

\begin{corollary}\label{cool_cor_1}
    Let $H$ be any graph, $n \in \mathbb{N}$ and $\mathrm{type} \in \{ \mathrm{ord}, \mathrm{ref}, \mathrm{cyc}, \mathrm{dih} \}$. Then:
    \begin{enumerate}[label=\textbf{(\arabic*)}]
        \item $R_\mathrm{type}(P^\mathrm{mon}_n, H) \le R_\mathrm{type}(P^\mathrm{mon}_{n + 1}, H)$;
        \item $R_\mathrm{type}(S^\mathrm{sc}_n, H) \le R_\mathrm{type}(S^\mathrm{sc}_{n + 1}, H)$; and
        \item $R_\mathrm{type}(K_n, H) \le R_\mathrm{type}(K_{n + 1}, H)$.
    \end{enumerate}
\end{corollary}
\begin{proof}
    The result follows from Lemma~\ref{cool_lemma} by observing that there exists an increasing injective homomorphism from $P^\mathrm{mon}_n$, $S^\mathrm{sc}_n$ and $K_n$, to $P^\mathrm{mon}_{n + 1}$, $S^\mathrm{sc}_{n + 1}$ and $K_{n + 1}$, respectively.
\end{proof}

\begin{corollary}
    Let $H$ be any graph, $n \in \mathbb{N}$ be even and $\mathrm{type} \in \{ \mathrm{ord}, \mathrm{ref}, \mathrm{cyc}, \mathrm{dih} \}$. Then
    $R_\mathrm{type}(M^\mathrm{nest}_n, H) \le R_\mathrm{type}(M^\mathrm{nest}_{n + 2}, H)$.
\end{corollary}
\begin{proof}
    The result follows from Lemma~\ref{cool_lemma} by observing that there exists an increasing injective homomorphism from $M^\mathrm{nest}_n$ to $M^\mathrm{nest}_{n + 2}$.
\end{proof}

\begin{corollary}\label{cool_cor_3}
    Let $H$ be any graph, $n \in \mathbb{N}$ and $\mathrm{type} \in \{ \mathrm{ref}, \mathrm{dih} \}$. Then $R_\mathrm{type}(P^\mathrm{alt}_n, H) \le R_\mathrm{type}(P^\mathrm{alt}_{n + 1}, H)$.
\end{corollary}
\begin{proof}
    Let $P^\mathrm{ralt}_{n + 1}$ denote the reverse alternating path of order $n + 1$, which arises by reflecting $P^\mathrm{alt}_{n + 1}$ with respect to the natural vertex order. Then there exists an increasing injective homomorphism from $P^\mathrm{alt}_n$ to $P^\mathrm{ralt}_{n + 1}$, so Lemma \ref{cool_lemma} gives $R_\mathrm{type}(P^\mathrm{alt}_n, H) \le R_\mathrm{type}(P^\mathrm{ralt}_{n + 1}, H)$. Since $P^\mathrm{ralt}_{n + 1}$ and $P^\mathrm{alt}_{n + 1}$ are equivalent in the context of reflective and dihedral Ramsey numbers, Proposition \ref{graph_invar} gives $R_\mathrm{type}(P^\mathrm{ralt}_{n + 1}, H) = R_\mathrm{type}(P^\mathrm{alt}_{n + 1}, H)$.
\end{proof}

\section{Methodology}\label{sc:meth}

In the present section, we describe how the computation of permutational Ramsey numbers can be carried out by solving SAT problems. Our approach is a direct generalization of the SAT-based approach originally proposed by Poljak for ordered Ramsey numbers \cite{Poljak2020} and recently extended to cyclic Ramsey numbers \cite{BaDamSteSto2026}.

Let $H_1$ and $H_2$ be two graphs, and let $\Gamma_1$ and $\Gamma_2$ be two permutation groups on their respective vertex sets. Now, consider a $2$-edge-coloring of $K_n$ for some $n \in \mathbb{N}$, and let $G_j$ be the spanning subgraph of $K_n$ comprising the edges of color $j$, for $j \in \{ 1, 2 \}$. For each $\{ u, v \} \in \binom{V(K_n)}{2}$, let $x_{\{ u, v \}}$ be a Boolean variable indicating whether the edge $\{ u, v \}$ is colored with color~$2$. Also, let $\mathrm{Inc}(V(H_1), V(K_n))$ and $\mathrm{Inc}(V(H_2), V(K_n))$ denote the sets of all increasing injections from $V(H_1)$ to $V(K_n)$ and from $V(H_2)$ to $V(K_n)$, respectively.

In the context of computing the permutational Ramsey number $R(H_1^{\Gamma_1}, H_2^{\Gamma_2})$, the search for a $2$-edge-coloring of $K_n$ containing no forbidden embeddings can be encoded as a SAT problem by observing that $H_1$ is not $\Gamma_1$-embeddable in $G_1$ if and only if
\begin{equation}\label{cl1}
    \bigwedge_{\psi \in \mathrm{Inc}(V(H_1), V(K_n))} \bigwedge_{\varphi \in \Gamma_1} \bigvee_{ \{ u, v \} \in E(H_1)} x_{\{ (\psi \circ \varphi)(u), (\psi \circ \varphi)(v) \}} ,
\end{equation}
while $H_2$ is not $\Gamma_2$-embeddable in $G_2$ if and only if
\begin{equation}\label{cl2}
    \bigwedge_{\psi \in \mathrm{Inc}(V(H_2), V(K_n))} \bigwedge_{\varphi \in \Gamma_2} \bigvee_{ \{ u, v \} \in E(H_2)} \neg x_{\{ (\psi \circ \varphi)(u), (\psi \circ \varphi)(v) \}} .
\end{equation}
Therefore, by combining \eqref{cl1} and \eqref{cl2}, we obtain a SAT problem with $\binom{n}{2}$ variables and $\binom{n}{|H_1|} |\Gamma_1| + \binom{n}{|H_2|} |\Gamma_2|$ clauses. If the SAT problem is satisfiable, then there exists a $2$-edge-coloring of $K_n$ containing no forbidden embeddings, hence $R(H_1^{\Gamma_1}, H_2^{\Gamma_2}) \ge n + 1$. Otherwise, such a $2$-edge-coloring of $K_n$ does not exist, implying that $R(H_1^{\Gamma_1}, H_2^{\Gamma_2}) \le n$. Therefore, solving the formulated SAT problem gives either a lower bound or an upper bound on $R(H_1^{\Gamma_1}, H_2^{\Gamma_2})$. The permutational Ramsey number can then be exactly computed by solving SAT instances for sufficiently many values of $n$.

Our computation of small reflective and dihedral Ramsey numbers of interest is automated through a pipeline of four \texttt{Python} modules:
\begin{enumerate}[label=\textbf{(\arabic*)}]
    \item \texttt{sat\_problem\_generator.py};
    \item \texttt{instance\_solver.py};
    \item \texttt{number\_finder.py}; and
    \item \texttt{grid\_explorer.py}.
\end{enumerate}
The first module, \texttt{sat\_problem\_generator.py}, is essentially the \texttt{src/sat\_solving/cnf\_generator.py} script from \cite{BaDamSteSto2026Repo}, with small modifications made to provide a suitable API for the other modules. The other three modules are implemented from scratch and their core logic is based on the \texttt{bash} scripts in the \texttt{src/sat\_solving} folder from \cite{BaDamSteSto2026Repo}. Our full source code is available at \cite{GitHub}, together with supplementary data, including the computed Ramsey numbers, the Kissat output files and all the $2$-edge-colorings yielding the lower bounds.

The \texttt{sat\_problem\_generator.py} script accepts two graphs $H_1$ and $H_2$ from the graph classes of interest, together with $n \in \mathbb{N}$ and a string indicating which type of Ramsey number is being considered. It then generates the corresponding SAT problem encoding the search for a $2$-edge-coloring of $K_n$ containing no forbidden embeddings. The generated SAT instance is printed out in the standard DIMACS CNF format, with the Boolean variables $x_{\{ u, v \}}$ arranged in ascending order first by $\min \{ u, v \}$, and then by $\max \{ u, v \}$, i.e., as
\[
    x_{\{0, 1\}}, x_{\{0, 2\}}, \ldots, x_{\{0, n - 1\}}, x_{\{1, 2\}}, x_{\{1, 3\}}, \ldots, x_{\{1, n - 1\}}, x_{\{2, 3\}}, x_{\{2, 4\}}, \ldots, x_{\{2, n - 1\}}, \ldots, x_{\{n - 2, n - 1\}}.
\]
The script is executed with six required arguments using the following syntax:
\begin{lstlisting}[language = bash, frame = trBL, escapeinside={(*@}{@*)}, aboveskip=10pt, belowskip=10pt, numbers=none, rulecolor=\color{black}]
python sat_problem_generator.py <class_a> <a> <class_b> <b> <problem_type> <n>
\end{lstlisting}

Here, \texttt{class\_a} and \texttt{class\_b} are strings from the set $\{ \mathrm{pmon}, \mathrm{palt}, \mathrm{cmon}, \mathrm{ssc}, \mathrm{k}, \mathrm{mnest} \}$ that specify the graph classes of $H_1$ and $H_2$, respectively, while \texttt{a} and \texttt{b} determine $|H_1|$ and $|H_2|$. Furthermore, \texttt{problem\_type} is a string from the set $\{ \mathrm{ord}, \mathrm{ref}, \mathrm{cyc}, \mathrm{dih} \}$ indicating which of the four supported types of Ramsey numbers (ordered, reflective, cyclic and dihedral) is being considered, while \texttt{n} represents $n$.

\begin{example}
Consider the following command:
\begin{lstlisting}[language = bash, frame = trBL, escapeinside={(*@}{@*)}, aboveskip=10pt, belowskip=10pt, numbers=none, rulecolor=\color{black}]
python sat_problem_generator.py palt 4 pmon 5 ref 12
\end{lstlisting}
This command prints out the DIMACS CNF representation of the SAT instance that encodes the search for a $2$-edge-coloring of $K_{12}$ containing no forbidden embeddings in the context of the reflective Ramsey number $R_\mathrm{ref}(P^\mathrm{alt}_{4}, P^\mathrm{mon}_5)$. \hfill$\Diamond$
\end{example}

The \texttt{instance\_solver.py} module solves a SAT instance encoding the search for a $2$-edge-coloring of $K_n$ containing no forbidden embeddings. It accepts the same arguments as the \texttt{sat\_problem\_generator.py} script, which is immediately invoked to generate an output DIMACS CNF file representing the generated SAT problem. Afterwards, the Kissat SAT solver is used on the output file in order to solve the generated SAT instance. Both of these steps have a time limit, and the execution is aborted in case of a timeout. Additionally, if the Kissat executable is not present, an auxiliary \texttt{bash} script from \cite{BaDamSteSto2026Repo} is used to obtain it automatically. To avoid trivial cases where the Ramsey number is identically one, both argument graphs are required to have at least two vertices.

The main function of the \texttt{instance\_solver.py} module returns an ordered pair \texttt{(status, graph)}, as follows.
\begin{enumerate}[label=\textbf{(\arabic*)}]
    \item If a timeout is reached, then both \texttt{status} and \texttt{graph} are set to \texttt{None}. In this case, it is not known whether the SAT instance has a solution.
    \item If the SAT instance is deemed unsatisfiable, then \texttt{status} is set to \texttt{False}, while \texttt{graph} is set to \texttt{None}.
    \item If the SAT instance is deemed satisfiable, then \texttt{status} is set to \texttt{True}, while \texttt{graph} is set to the \texttt{graph6} \cite{McKayGraph6, McKayPip2014} string representation of the solution found to the SAT instance. Here, the presence of an edge indicates that its color in $K_n$ is $2$, i.e., that the corresponding Boolean variable is \texttt{True}.
\end{enumerate}
In addition, the \texttt{instance\_solver.py} module saves the Kissat outputs to separate files in the \texttt{kissat\_output} folder and, in the case of satisfiable instances, saves the \texttt{graph6} string representations of the corresponding solutions to separate files in the \texttt{parsed\_graphs} folder.

\begin{example}
Consider the following command:
\begin{lstlisting}[language = bash, frame = trBL, escapeinside={(*@}{@*)}, aboveskip=10pt, belowskip=10pt, numbers=none, rulecolor=\color{black}]
python instance_solver.py palt 4 pmon 5 ref 12
\end{lstlisting}
This command gives the following output:
\begin{lstlisting}[language = {}, frame = trBL, escapeinside={(*@}{@*)}, aboveskip=10pt, belowskip=10pt, numbers=none, rulecolor=\color{black}]
(True, 'KCrffr{~f}^{')
\end{lstlisting}
Therefore, the corresponding SAT instance is satisfiable, with the solution given as a \texttt{graph6} string. Hence, $R_\mathrm{ref}(P^\mathrm{alt}_4, P^\mathrm{mon}_5) \ge 13$. Now, consider the following command:
\begin{lstlisting}[language = bash, frame = trBL, escapeinside={(*@}{@*)}, aboveskip=10pt, belowskip=10pt, numbers=none, rulecolor=\color{black}]
python instance_solver.py palt 4 pmon 5 ref 13
\end{lstlisting}
Its execution gives the following output:
\begin{lstlisting}[language = bash, frame = trBL, escapeinside={(*@}{@*)}, aboveskip=10pt, belowskip=10pt, numbers=none, rulecolor=\color{black}]
(False, None)
\end{lstlisting}
This means that the corresponding SAT instance is unsatisfiable, implying $R_\mathrm{ref}(P^\mathrm{alt}_4, P^\mathrm{mon}_5) \le 13$. With these two executions in mind, we conclude that $R_\mathrm{ref}(P^\mathrm{alt}_4, P^\mathrm{mon}_5) = 13$. \hfill$\Diamond$
\end{example}

Similar to the previous two \texttt{Python} files, the \texttt{number\_finder.py} module accepts two graphs $H_1$ and $H_2$ from the graph classes of interest, satisfying $|H_1|, |H_2| \ge 2$, together with a string indicating which type of Ramsey number is being considered. It then computes the considered Ramsey number by solving SAT instances via the \texttt{instance\_solver.py} module for sufficiently many values of $n$. The \texttt{number\_finder.py} module can be executed using the following syntax:
\begin{lstlisting}[language = bash, frame = trBL, escapeinside={(*@}{@*)}, aboveskip=10pt, belowskip=10pt, numbers=none, rulecolor=\color{black}]
python number_finder.py <class_a> <a> <class_b> <b> <problem_type> [<starting_n>]
\end{lstlisting}

The first five arguments are required and have the same meaning as in the previous two modules, while the \texttt{starting\_n} argument is optional. If this argument is not provided, then the script attempts to generate and solve a SAT instance for each integer $n$ starting from the trivial lower bound $\max \{ |H_1|, |H_2| \}$, until the first unsatisfiable SAT instance is encountered or a timeout is reached. If \texttt{starting\_n} is provided and exceeds $\max \{ |H_1|, |H_2| \}$, then it is treated as a nontrivial lower bound on the considered Ramsey number, so the script attempts to generate and solve SAT instances for each integer $n$ starting from \texttt{starting\_n}. The user must guarantee that the provided 
\texttt{starting\_n} is a valid lower bound; otherwise, the execution will yield an incorrect result.

The main function of the \texttt{number\_finder.py} module returns an ordered pair \texttt{(exact, number)}, as follows.
\begin{enumerate}[label=\textbf{(\arabic*)}]
    \item If an unsatisfiable SAT instance is encountered for the first time, this means that the corresponding $n$ is equal to the considered Ramsey number. In this case, \texttt{exact} is set to \texttt{True}, while \texttt{number} is set to the computed value of the considered Ramsey number.
    \item If a timeout is reached, then the corresponding $n$ is a lower bound on the considered Ramsey number, and it is not known if equality holds. In this case, \texttt{exact} is set to \texttt{False}, while \texttt{number} is set to the obtained lower bound on the considered Ramsey number.
\end{enumerate}

\begin{example}
The reflective Ramsey number $R_\mathrm{ref}(P^\mathrm{alt}_4, P^\mathrm{mon}_5)$ can be computed by executing the following command:
\begin{lstlisting}[language = bash, frame = trBL, escapeinside={(*@}{@*)}, aboveskip=10pt, belowskip=10pt, numbers=none, rulecolor=\color{black}]
python number_finder.py palt 4 pmon 5 ref
\end{lstlisting}
Upon execution, the script prints out the following output:
\begin{lstlisting}[language = {}, frame = trBL, escapeinside={(*@}{@*)}, aboveskip=10pt, belowskip=10pt, numbers=none, rulecolor=\color{black}]
palt 4 vs. pmon 5 // ref
Inspecting for n = 5... Satisfiable!
Inspecting for n = 6... Satisfiable!
Inspecting for n = 7... Satisfiable!
Inspecting for n = 8... Satisfiable!
Inspecting for n = 9... Satisfiable!
Inspecting for n = 10... Satisfiable!
Inspecting for n = 11... Satisfiable!
Inspecting for n = 12... Satisfiable!
Inspecting for n = 13... Unsatisfiable!
(True, 13)
\end{lstlisting}
Therefore, $R_\mathrm{ref}(P^\mathrm{alt}_4, P^\mathrm{mon}_5) = 13$. \hfill$\Diamond$
\end{example}

Finally, the \texttt{grid\_explorer.py} module computes Ramsey numbers of a selected type with arguments from two provided graph classes. Although this \texttt{Python} module is based on the \texttt{bash} scripts in the \texttt{src/sat\_solving} folder from \cite{BaDamSteSto2026Repo}, its logic is significantly optimized to remove redundancy and improve the efficiency of extensive computation. The module can be executed using the following syntax:
\begin{lstlisting}[language = bash, frame = trBL, escapeinside={(*@}{@*)}, aboveskip=10pt, belowskip=10pt, numbers=none, rulecolor=\color{black}]
python grid_explorer.py <class_a> <class_b> <problem_type>
\end{lstlisting}
The three required arguments have the same meaning as in the previous modules. The \texttt{problem\_type} argument needs to be from the set $\{ \mathrm{ref}, \mathrm{dih} \}$ if at least one of the graph classes is the class of alternating paths, due to the use of Corollary \ref{cool_cor_3}.

The computation of Ramsey numbers is carried out via the \texttt{number\_finder.py} module by arranging the numbers in an infinite matrix in which the rows correspond to the orders of graphs from the first graph class, while the columns correspond to those from the second graph class. More precisely, the rows and columns correspond to the orders $3, 4, 5, \ldots$, if the graph class is not $M^\mathrm{nest}$, and $4, 6, 8, 10, \ldots$, otherwise. The entries are grouped by the minimum of their row and column indices in a layer-like manner, as shown in Table~\ref{collect_table}. The computation proceeds layer by layer, with the last processed layer being the one whose diagonal entry is the first to yield a timeout according to the configured SAT instance generation and Kissat time limits. Within each layer, the entries are computed from the diagonal entry onward to the right and downward. In each direction, the traversal is terminated after three occurrences of a timeout or once the graph order exceeds a configurable threshold. The computation of Ramsey numbers is optimized through the use of Corollaries~\ref{cool_cor_1}--\ref{cool_cor_3}. In addition, if the two graph classes corresponding to the rows and the columns are the same, then only the upper-triangular part of the infinite matrix, including the diagonal, is considered.

\begin{table}[h]
\centering
\renewcommand{\arraystretch}{1.75}
\setlength{\tabcolsep}{10pt}
\setlength{\extrarowheight}{-1.25pt}
\setlength{\arrayrulewidth}{0.55pt}
\begin{tabular}{ccccc}
    \cline{1-5}
    \multicolumn{1}{|c}{$R_\mathrm{dih}(P^\mathrm{mon}_3, M^\mathrm{nest}_4)$} & $R_\mathrm{dih}(P^\mathrm{mon}_3, M^\mathrm{nest}_6)$ & $R_\mathrm{dih}(P^\mathrm{mon}_3, M^\mathrm{nest}_8)$ & $R_\mathrm{dih}(P^\mathrm{mon}_3, M^\mathrm{nest}_{10})$ & $\cdots$\\ \cline{2-5}
    \multicolumn{1}{|c}{$R_\mathrm{dih}(P^\mathrm{mon}_4, M^\mathrm{nest}_4)$} & \multicolumn{1}{|c}{$R_\mathrm{dih}(P^\mathrm{mon}_4, M^\mathrm{nest}_6)$} & $R_\mathrm{dih}(P^\mathrm{mon}_4, M^\mathrm{nest}_8)$ & $R_\mathrm{dih}(P^\mathrm{mon}_4, M^\mathrm{nest}_{10})$ & $\cdots$ \\ \cline{3-5}
    \multicolumn{1}{|c}{$R_\mathrm{dih}(P^\mathrm{mon}_5, M^\mathrm{nest}_4)$} & \multicolumn{1}{|c}{$R_\mathrm{dih}(P^\mathrm{mon}_5, M^\mathrm{nest}_6)$} & \multicolumn{1}{|c}{$R_\mathrm{dih}(P^\mathrm{mon}_5, M^\mathrm{nest}_8)$} & $R_\mathrm{dih}(P^\mathrm{mon}_5, M^\mathrm{nest}_{10})$ & $\cdots$ \\ \cline{4-5}
    \multicolumn{1}{|c}{$R_\mathrm{dih}(P^\mathrm{mon}_6, M^\mathrm{nest}_4)$} & \multicolumn{1}{|c}{$R_\mathrm{dih}(P^\mathrm{mon}_6, M^\mathrm{nest}_6)$} & \multicolumn{1}{|c}{$R_\mathrm{dih}(P^\mathrm{mon}_6, M^\mathrm{nest}_8)$} & \multicolumn{1}{|c}{$R_\mathrm{dih}(P^\mathrm{mon}_6, M^\mathrm{nest}_{10})$} & $\cdots$ \\ \cline{5-5}
    \multicolumn{1}{|c}{$\vdots$} & \multicolumn{1}{|c}{$\vdots$} & \multicolumn{1}{|c}{$\vdots$} & \multicolumn{1}{|c}{$\vdots$} & \multicolumn{1}{|c}{$\ddots$} \\
\end{tabular}
\caption{The infinite matrix containing all Ramsey numbers of the form $R_\mathrm{dih}(P_a^\mathrm{mon}, M_b^\mathrm{nest})$, with entries grouped by the minimum of their row and column indices.}
\label{collect_table}
\end{table}

The \texttt{grid\_explorer.py} module also saves the computed results to a CSV file with the header \texttt{a, b, exact, number}. Here, the columns \texttt{a} and \texttt{b} specify the graph orders of the two arguments, while \texttt{number} stores the computed value. The \texttt{exact} column contains a Boolean flag indicating whether the value in the \texttt{number} column has been determined to be the exact Ramsey number or is only known to be a lower bound. The auxiliary script \texttt{table\_generator.py}, whose source code is mostly based on the \texttt{src/sat\_solving/kissat\_output\_parser.py} script from \cite{BaDamSteSto2026Repo}, can then be used to read all the saved CSV files and automatically generate \LaTeX\ code that displays the computed results in tables.

\begin{example}
Consider the following command:
\begin{lstlisting}[language = bash, frame = trBL, escapeinside={(*@}{@*)}, aboveskip=10pt, belowskip=10pt, numbers=none, rulecolor=\color{black}]
python grid_explorer.py palt pmon ref
\end{lstlisting}
This command can be used to automate the computation of Ramsey numbers of the form $R_\mathrm{ref}(P^\mathrm{alt}_a, P^\mathrm{mon}_b)$ with $a, b \ge 3$. Next, consider the following command:
\begin{lstlisting}[language = bash, frame = trBL, escapeinside={(*@}{@*)}, aboveskip=10pt, belowskip=10pt, numbers=none, rulecolor=\color{black}]
python table_generator.py
\end{lstlisting}
The execution of this command scans the saved CSV file and automatically generates the \LaTeX\ code for the corresponding table. \hfill$\Diamond$
\end{example}

\section{Results}\label{sc:results}

In this section, we present the computational results obtained by executing the software described in Section~\ref{sc:meth}, together with several general results and conjectures motivated by the computational findings. With these results and conjectures in mind, Table \ref{summary_tab} summarizes the known and conjectured formulas for the ordered, reflective, cyclic and dihedral Ramsey numbers of the considered combinations of graph classes. All the computational results presented below were obtained by executing the software on an Apple MacBook Pro (Apple~M4~Pro, 24~GB RAM, macOS Sequoia 15.7.5), while the \LaTeX\ source code for Tables \ref{palt_palt_ref_tab}--\ref{ssc_mnest_ref_tab} was automatically generated by the \texttt{table\_generator.py} script. The time limit for generating SAT instances was set to 2 minutes, and the time limit for the Kissat SAT solver was set to 3 minutes. When using the \texttt{grid\_explorer.py} script, an order threshold of 30 was employed for both arguments.

We begin with a theorem that completely determines the permutational Ramsey numbers in which one argument is a monotone path with a trivial permutation group, while the other is any connected graph with an arbitrary permutation group. The upper bound was implicitly shown by Károlyi, Pach and Tóth \cite{KaPaTo1997}, as remarked by Balko, Cibulka, Král and Kynčl \cite[Lemma~18]{BalCiKralKyn2020}, while the lower bound can be established similarly to \cite[Theorem~4.1]{BaDamSteSto2026}. For completeness, we provide the full proof.

\begin{table}[t]
\centering
\resizebox{\textwidth}{!}{
\begin{tabular}{|c|c|c|c|c|}
\hline
Argument $1$ & Argument $2$ & Type & Status & Formula\\
\hline
\hline
\multirow{4}{*}{alternating path} & \multirow{4}{*}{alternating path} & ordered & --- & --- \\
\cline{3-5}
& & reflective & --- & --- \\
\cline{3-5}
& & cyclic & conjectured \cite[Conjecture~4.8]{BaDamSteSto2026} & $a + b - 2 - (ab \bmod 2)$ for $a, b \ge 2$\\
\cline{3-5}
& & dihedral & conjectured (Conjecture~\ref{palt_palt_dih_conj}) & $a + b - 2 - (ab \bmod 2)$ for $a, b \ge 2$\\
\hline
\multirow{4}{*}{alternating path} & \multirow{4}{*}{start-central star} & ordered & --- & --- \\
\cline{3-5}
& & reflective & --- & --- \\
\cline{3-5}
& & cyclic & conjectured \cite[Conjecture~4.17]{BaDamSteSto2026} & $a + b - 2 - (ab \bmod 2)$ for $a, b \ge 2$\\
\cline{3-5}
& & dihedral & conjectured (Conjecture \ref{palt_ssc_dih_conj}) & $a + b - 2 - (ab \bmod 2)$ for $a, b \ge 2$\\
\hline
\multirow{4}{*}{alternating path} & \multirow{4}{*}{monotone path} & ordered & solved \cite[Corollary~4.2]{BaDamSteSto2026} & $1 + (a - 1)(b - 1)$ for $a, b \in \mathbb{N}$\\
\cline{3-5}
& & reflective & solved (Corollary~\ref{palt_pmon_ref_cor}) & $1 + (a - 1)(b - 1)$ for $a, b \in \mathbb{N}$\\
\cline{3-5}
& & cyclic & conjectured \cite[Conjecture~4.9]{BaDamSteSto2026} & $1 + (a - 1)(b - 2)$ for $a \ge 3$, $b \ge 4$\\
\cline{3-5}
& & dihedral & conjectured (Conjecture \ref{palt_pmon_dih_conj}) & $1 + (a - 1)(b - 2)$ for $a \ge 3$, $b \ge 4$\\
\hline
\multirow{4}{*}{alternating path} & \multirow{4}{*}{monotone cycle} & ordered & conjectured \cite[Conjecture~4.13]{BaDamSteSto2026} & $\lceil (a - 1)(b - \frac{1}{2}) \rceil$ for $a \ge 2$, $b \ge 3$ \\
\cline{3-5}
& & reflective & conjectured (Conjecture \ref{palt_cmon_ref_conj}) & $\lfloor (a - 1)(b - \frac{1}{2}) \rfloor$ for $a, b \ge 3$\\
\cline{3-5}
& & cyclic & conjectured \cite[Conjecture~4.14]{BaDamSteSto2026} & $1 + (a - 1)(b - 1)$ for $a \in \mathbb{N}$, $b \ge 2$\\
\cline{3-5}
& & dihedral & conjectured (Conjecture \ref{palt_cmon_dih_conj}) & $1 + (a - 1)(b - 1)$ for $a \in \mathbb{N}$, $b \ge 2$\\
\hline
\multirow{4}{*}{alternating path} & \multirow{4}{*}{complete graph} & ordered & conjectured \cite[Conjecture~4.22]{BaDamSteSto2026} & $\lceil \frac{3ab - 5b}{2} \rceil - 2a + 5$ for $a \ge 3$, $b \ge 2$\\
\cline{3-5}
& & reflective & --- & --- \\
\cline{3-5}
& & cyclic & conjectured \cite[Conjecture~4.23]{BaDamSteSto2026} & $1 + (a - 1)(b - 1)$ for $a, b \in \mathbb{N}$\\
\cline{3-5}
& & dihedral & conjectured (Conjecture \ref{palt_k_dih_conj}) & $1 + (a - 1)(b - 1)$ for $a, b \in \mathbb{N}$\\
\hline
\multirow{4}{*}{alternating path} & \multirow{4}{*}{nested matching} & ordered & --- & --- \\
\cline{3-5}
& & reflective & --- & --- \\
\cline{3-5}
& & cyclic & --- & --- \\
\cline{3-5}
& & dihedral & --- & --- \\
\hline
\multirow{3}{*}{start-central star} & \multirow{3}{*}{start-central star} & ordered & solved \cite[Observation~12]{BalCiKralKyn2020} & $a + b - 2$ for $a, b \ge 2$\\
\cline{3-5}
& & reflective & solved (Proposition \ref{ssc_ssc_ref_prop}) & $a + b - 2$ for $a, b \ge 2$ \\
\cline{3-5}
& & cyclic & solved \cite{BurrRo1973} & $a + b - 2 - (ab \bmod 2)$ for $a, b \ge 2$\\
\hline
\multirow{3}{*}{start-central star} & \multirow{3}{*}{monotone path} & ordered & solved \cite[Corollary~4.4]{BaDamSteSto2026} & $1 + (a - 1)(b - 1)$ for $a, b \in \mathbb{N}$\\
\cline{3-5}
& & reflective & solved (Corollary \ref{ssc_pmon_ref_cor}) & $1 + (a - 1)(b - 1)$ for $a, b \in \mathbb{N}$ \\
\cline{3-5}
& & cyclic & conjectured \cite[Conjecture~4.16]{BaDamSteSto2026} & $1 + (a - 1)(b - 2)$ for $a \ge 3$, $b \ge 4$\\
\hline
\multirow{3}{*}{start-central star} & \multirow{3}{*}{monotone cycle} & ordered & solved \cite[Theorem~4.18]{BaDamSteSto2026} & $1 + (a - 1)(b - 1)$ for $a \in \mathbb{N}$, $b \ge 2$ \\
\cline{3-5}
& & reflective & solved (Proposition \ref{ssc_cmon_ref_prop}) & $1 + (a - 1)(b - 1)$ for $a \in \mathbb{N}$, $b \ge 2$\\
\cline{3-5}
& & cyclic & solved \cite[Corollary~4.19]{BaDamSteSto2026} & $1 + (a - 1)(b - 1)$ for $a \in \mathbb{N}$, $b \ge 2$ \\
\hline
\multirow{3}{*}{start-central star} & \multirow{3}{*}{complete graph} & ordered & solved \cite[Theorem~4.24]{BaDamSteSto2026} & $1 + (a - 1)(b - 1)$ for $a, b \in \mathbb{N}$\\
\cline{3-5}
& & reflective & solved (Corollary \ref{ssc_k_ref_cor}) & $1 + (a - 1)(b - 1)$ for $a, b \in \mathbb{N}$ \\
\cline{3-5}
& & cyclic & solved \cite{Chvatal1977} & $1 + (a - 1)(b - 1)$ for $a, b \in \mathbb{N}$\\
\hline
\multirow{3}{*}[-12.3pt]{start-central star} & \multirow{3}{*}[-12.3pt]{nested matching} & ordered & --- & --- \\
\cline{3-5}
& & reflective & --- & --- \\
\cline{3-5}
& & cyclic & conjectured \cite[Conjecture~4.33]{BaDamSteSto2026} & \rule{0pt}{4.7ex} $\begin{cases}
            a + b - 3, & \mbox{if $2 \nmid a$ and $4 \mid b$},\\
            a + b - 2, & \mbox{otherwise}
        \end{cases}$ for $a \ge 2$ and even $b \ge 2$ \rule[-3.5ex]{0pt}{0pt} \\
\hline
\end{tabular}
}
\caption{The known and conjectured formulas for the ordered, reflective, cyclic and dihedral Ramsey numbers of alternating paths and start-central stars versus alternating paths, start-central stars, monotone paths, monotone cycles, complete graphs and nested matchings. The dihedral Ramsey numbers of start-central stars versus start-central stars, monotone paths, monotone cycles, complete graphs and nested matchings are omitted, since they coincide with the corresponding cyclic Ramsey numbers. In the Formula column, $a$ and $b$ denote the orders of the first and second graph arguments, respectively. Entries marked ``---'' indicate that no general formula is currently known or conjectured.}
\label{summary_tab}
\end{table}

\begin{theorem}\label{brutal_th}
    For $a, b \in \mathbb{N}$, let $H$ be any connected graph of order $a$, $\Gamma$ any permutation group on $V(H)$, and $\Lambda$ the trivial permutation group on $V(P^\mathrm{mon}_b)$. Then $R\left( H^\Gamma, (P_b^\mathrm{mon})^\Lambda \right) = 1 + (a - 1)(b - 1)$.
\end{theorem}
\begin{proof}
    The theorem trivially holds if $\min \{ a, b \} = 1$, so we assume that $a, b \ge 2$. First, we establish the upper bound by showing that every $2$-edge-coloring of $K_n$ with $n = 1 + (a - 1)(b - 1)$ contains a forbidden embedding. We carry out the proof by induction on $b$. To begin, the statement trivially holds for $b = 2$, so we assume that $b \ge 3$. Consider a $2$-edge-coloring of $K_n$, and let $G_j$ denote the spanning subgraph of $K_n$ comprising the edges of color~$j$, for $j \in \{ 1, 2 \}$. Suppose that $H$ is not $\Gamma$-embeddable in $G_1$. Then, since $\left( 1 + (a - 1)(b - 1) \right) - \left( 1 + (a - 1)(b - 2) \right) = a - 1$, by induction, there exist $a$ increasing injective homomorphisms $\varphi_1, \varphi_2, \ldots, \varphi_a$ from $P^\mathrm{mon}_{b - 1}$ to $G_2$ such that the numbers $\varphi_1(b - 2), \varphi_2(b - 2), \ldots, \varphi_a(b - 2)$ are all mutually distinct. If none of the edges between the vertices $\varphi_1(b - 2), \varphi_2(b - 2), \ldots, \varphi_a(b - 2)$ is colored in color $2$, then there exists an increasing injective homomorphism from $H$ to $G_1$, yielding a contradiction. Thus, at least one of these edges is colored in color $2$, implying the existence of an increasing injective homomorphism from $P^\mathrm{mon}_b$ to $G_2$. Therefore, $P^\mathrm{mon}_b$ is $\Lambda$-embeddable in $G_2$, and hence $R\left( H^\Gamma, (P_b^\mathrm{mon})^\Lambda \right) \le 1 + (a - 1)(b - 1)$.

    Next, we establish the lower bound by constructing a $2$-edge-coloring of $K_n$ with $n = (a - 1)(b - 1)$ containing no forbidden embeddings. Let $f \colon \{ 0, 1, 2, \ldots, n - 1 \} \to \{ 0, 1, 2, \ldots, b - 2 \}$ be the function defined by $f(x) \coloneqq \lfloor \frac{x}{a - 1} \rfloor$, and consider the $2$-edge-coloring of $K_n$ in which the edge between any two distinct vertices $u, v \in V(K_n)$ is colored with color~$1$ if and only if $f(u) = f(v)$. Let $G_1$ and $G_2$ be defined as above. By way of contradiction, suppose that $H$ is $\Gamma$-embeddable in $G_1$. Then there exists an injective homomorphism from $H$ to $G_1$. Since $H$ is connected, this means that some component of $G_1$ contains at least $a$ vertices, which yields a contradiction because all components of $G_1$ are isomorphic to $K_{a - 1}$.

    Now, by way of contradiction, suppose that $P^\mathrm{mon}_b$ is $\Lambda$-embeddable in $G_2$. Then there exists an increasing injective homomorphism $\varphi$ from $P^\mathrm{mon}_b$ to $G_2$. Since $f$ is nondecreasing, and $f(\varphi(v)) \neq f(\varphi(v + 1))$ holds for every $v \in \{ 0, 1, 2, \ldots, b - 2 \}$, it follows that the function $f \circ \varphi \colon \{ 0, 1, 2, \ldots, b - 1 \} \to \{ 0, 1, 2, \ldots, b - 2 \}$ is strictly increasing, which is clearly impossible. Therefore, $R\left( H^\Gamma, (P_b^\mathrm{mon})^\Lambda \right) \ge 1 + (a - 1)(b - 1)$.
\end{proof}

\noindent
The following results are immediate from Theorem \ref{brutal_th}, Proposition \ref{equiv_class} and the fact that $P^\mathrm{mon}_b$ has reflection symmetry.

\begin{corollary}\label{palt_pmon_ref_cor}
    For any $a, b \in \mathbb{N}$, we have $R_\mathrm{ref}(P^\mathrm{alt}_a, P^\mathrm{mon}_b) = 1 + (a - 1)(b - 1)$.
\end{corollary}
\begin{corollary}\label{ssc_pmon_ref_cor}
    For any $a, b \in \mathbb{N}$, we have $R_\mathrm{ref}(S^\mathrm{sc}_a, P^\mathrm{mon}_b) = 1 + (a - 1)(b - 1)$.
\end{corollary}

We proceed by studying the reflective and dihedral Ramsey numbers of alternating paths versus alternating paths. Our computational results are given in Tables \ref{palt_palt_ref_tab} and \ref{palt_palt_dih_tab}. As it turns out, the reflective Ramsey numbers are similar, but not identical, to the corresponding ordered Ramsey numbers, and they exhibit a pattern that is difficult to discern. On the other hand, the dihedral Ramsey numbers appear to coincide with the cyclic Ramsey numbers, leading to the following conjecture.

\begin{conjecture}\label{palt_palt_dih_conj}
    For any $a, b \ge 2$, we have $R_\mathrm{dih}(P_a^\mathrm{alt}, P_b^\mathrm{alt}) = a + b - 2 - (ab \bmod 2)$.
\end{conjecture}

\begin{table}[H]
\centering
\resizebox{\textwidth}{!}{
\begin{tabular}{|c||rrrrrrrrrrrrrrrrrrrrrrrrr|}
\hline
\backslashbox{$a$}{$b$} & $3$ & $4$ & $5$ & $6$ & $7$ & $8$ & $9$ & $10$ & $11$ & $12$ & $13$ & $14$ & $15$ & $16$ & $17$ & $18$ & $19$ & $20$ & $21$ & $22$ & $23$ & $24$ & $25$ & $26$ & $27$\\
\hline
\hline
$3$ & $4$ & $6$ & $7$ & $8$ & $9$ & $11$ & $12$ & $13$ & $14$ & $15$ & $16$ & $18$ & $19$ & $20$ & $21$ & $22$ & $23$ & $24$ & $25$ & $27$ & $28$ & $29$ & $\ge 30$ & $\ge 31$ & $\ge 32$\\
$4$ & & $7$ & $8$ & $10$ & $11$ & $12$ & $14$ & $15$ & $16$ & $17$ & $18$ & $20$ & $21$ & $22$ & $\ge 23$ & $\ge 24$ & $\ge 25$ & & & & & & & &\\
$5$ & & & $9$ & $11$ & $12$ & $14$ & $15$ & $17$ & $18$ & $19$ & $\ge 20$ & $\ge 21$ & $\ge 23$ & & & & & & & & & & & &\\
$6$ & & & & $12$ & $13$ & $15$ & $16$ & $18$ & $\ge 19$ & $\ge 21$ & $\ge 22$ & & & & & & & & & & & & & &\\
$7$ & & & & & $14$ & $16$ & $18$ & $\ge 19$ & $\ge 21$ & $\ge 22$ & & & & & & & & & & & & & & &\\
$8$ & & & & & & $17$ & $\ge 19$ & $\ge 20$ & $\ge 22$ & & & & & & & & & & & & & & & &\\
$9$ & & & & & & & $\ge 20$ & $\ge 21$ & $\ge 23$ & & & & & & & & & & & & & & & &\\
\hline
\end{tabular}
}
\caption{Reflective Ramsey numbers $R_\mathrm{ref}(P^\mathrm{alt}_a, P^\mathrm{alt}_b)$ with $3 \le a \le b$, $a \le 9$ and $b \le 27$.}
\label{palt_palt_ref_tab}
\end{table}

\begin{table}[H]
\centering
\resizebox{\textwidth}{!}{
\begin{tabular}{|c||rrrrrrrrrrrrrrrrrrrrrrrr|}
\hline
\backslashbox{$a$}{$b$} & $3$ & $4$ & $5$ & $6$ & $7$ & $8$ & $9$ & $10$ & $11$ & $12$ & $13$ & $14$ & $15$ & $16$ & $17$ & $18$ & $19$ & $20$ & $21$ & $22$ & $23$ & $24$ & $25$ & $26$\\
\hline
\hline
$3$ & $3$ & $5$ & $5$ & $7$ & $7$ & $9$ & $9$ & $11$ & $11$ & $13$ & $13$ & $15$ & $15$ & $17$ & $17$ & $19$ & $19$ & $21$ & $21$ & $23$ & $\ge 23$ & $25$ & $\ge 25$ & $\ge 27$\\
$4$ & & $6$ & $7$ & $8$ & $9$ & $10$ & $11$ & $12$ & $13$ & $14$ & $15$ & $16$ & $\ge 17$ & $\ge 18$ & $\ge 19$ & & & & & & & & &\\
$5$ & & & $7$ & $9$ & $9$ & $11$ & $11$ & $13$ & $\ge 13$ & $\ge 15$ & $\ge 15$ & & & & & & & & & & & & &\\
$6$ & & & & $10$ & $11$ & $12$ & $13$ & $\ge 14$ & $\ge 15$ & $\ge 16$ & & & & & & & & & & & & & &\\
$7$ & & & & & $\ge 11$ & $13$ & $\ge 13$ & $\ge 15$ & & & & & & & & & & & & & & & &\\
\hline
\end{tabular}
}
\caption{Dihedral Ramsey numbers $R_\mathrm{dih}(P^\mathrm{alt}_a, P^\mathrm{alt}_b)$ with $3 \le a \le b$, $a \le 7$ and $b \le 26$.}
\label{palt_palt_dih_tab}
\end{table}

Now, consider the computational results for the reflective and dihedral Ramsey numbers of alternating paths versus start-central stars; see Tables \ref{palt_ssc_ref_tab} and \ref{palt_ssc_dih_tab}. Once again, the reflective Ramsey numbers appear to be similar, but not identical, to the corresponding ordered Ramsey numbers, and their structure seems difficult to understand. An interesting observation about the dihedral Ramsey numbers is that exact values were obtained for all the numbers in the first row of Table~\ref{palt_ssc_dih_tab} with $b \le 18$, whereas in the second row exact values were obtained for all $b \le 30$. The results suggest that Kissat unexpectedly struggles with the case $a = 3$ and odd $b \ge 19$. Nonetheless, the dihedral Ramsey numbers appear to coincide with the cyclic Ramsey numbers, which leads to the following conjecture.

\begin{conjecture}\label{palt_ssc_dih_conj}
    For any $a, b \ge 2$, we have $R_\mathrm{dih}(P^\mathrm{alt}_a, S^\mathrm{sc}_b) = a + b - 2 - (ab \bmod 2)$.
\end{conjecture}

The reflective Ramsey numbers of alternating paths versus monotone paths are fully determined by Corollary~\ref{palt_pmon_ref_cor}, and they coincide with the corresponding ordered Ramsey numbers; see Table~\ref{summary_tab}. The dihedral Ramsey numbers also appear to coincide with the cyclic Ramsey numbers, as suggested by the computational results in Table \ref{palt_pmon_dih_tab}. Thus, we pose the following conjecture.

\begin{conjecture}\label{palt_pmon_dih_conj}
    For any $a \ge 3$ and $b \ge 4$, we have $R_\mathrm{dih}(P_a^\mathrm{alt}, P_b^\mathrm{mon}) = 1 + (a - 1)(b - 2)$.
\end{conjecture}

\newpage

\begin{table}[H]
\centering
\resizebox{\textwidth}{!}{
\begin{tabular}{|c||rrrrrrrrrrrrrrrrrrrrrrrrrrrr|}
\hline
\backslashbox{$a$}{$b$} & $3$ & $4$ & $5$ & $6$ & $7$ & $8$ & $9$ & $10$ & $11$ & $12$ & $13$ & $14$ & $15$ & $16$ & $17$ & $18$ & $19$ & $20$ & $21$ & $22$ & $23$ & $24$ & $25$ & $26$ & $27$ & $28$ & $29$ & $30$\\
\hline
\hline
$3$ & $4$ & $5$ & $6$ & $7$ & $8$ & $9$ & $10$ & $11$ & $12$ & $13$ & $14$ & $15$ & $16$ & $17$ & $18$ & $19$ & $20$ & $21$ & $22$ & $23$ & $24$ & $25$ & $26$ & $27$ & $28$ & $29$ & $30$ & $31$\\
$4$ & $6$ & $7$ & $8$ & $10$ & $11$ & $12$ & $13$ & $15$ & $16$ & $17$ & $18$ & $19$ & $21$ & $22$ & $23$ & $\ge 24$ & $\ge 25$ & $\ge 26$ & & & & & & & & & &\\
$5$ & $7$ & $8$ & $10$ & $11$ & $12$ & $14$ & $15$ & $16$ & $17$ & $\ge 19$ & $\ge 20$ & $\ge 21$ & & & & & & & & & & & & & & & &\\
$6$ & $8$ & $10$ & $11$ & $13$ & $14$ & $16$ & $17$ & $\ge 18$ & $\ge 19$ & $\ge 21$ & & & & & & & & & & & & & & & & & &\\
$7$ & $9$ & $11$ & $13$ & $14$ & $15$ & $\ge 17$ & $\ge 18$ & $\ge 20$ & & & & & & & & & & & & & & & & & & & &\\
$8$ & $11$ & $12$ & $14$ & $16$ & $\ge 17$ & $\ge 19$ & $\ge 20$ & $\ge 21$ & & & & & & & & & & & & & & & & & & & &\\
$9$ & $12$ & $14$ & $15$ & $17$ & $\ge 18$ & $\ge 20$ & & & & & & & & & & & & & & & & & & & & & &\\
$10$ & $13$ & $15$ & $17$ & $\ge 19$ & $\ge 20$ & $\ge 22$ & & & & & & & & & & & & & & & & & & & & & &\\
$11$ & $14$ & $16$ & $18$ & $\ge 20$ & & & & & & & & & & & & & & & & & & & & & & & &\\
$12$ & $15$ & $18$ & $\ge 19$ & $\ge 21$ & & & & & & & & & & & & & & & & & & & & & & & &\\
$13$ & $16$ & $19$ & $\ge 20$ & & & & & & & & & & & & & & & & & & & & & & & & &\\
$14$ & $18$ & $20$ & $\ge 22$ & & & & & & & & & & & & & & & & & & & & & & & & &\\
$15$ & $19$ & $\ge 21$ & & & & & & & & & & & & & & & & & & & & & & & & & &\\
$16$ & $20$ & $\ge 22$ & & & & & & & & & & & & & & & & & & & & & & & & & &\\
$17$ & $21$ & $\ge 23$ & & & & & & & & & & & & & & & & & & & & & & & & & &\\
$18$ & $22$ & & & & & & & & & & & & & & & & & & & & & & & & & & &\\
$19$ & $23$ & & & & & & & & & & & & & & & & & & & & & & & & & & &\\
$20$ & $24$ & & & & & & & & & & & & & & & & & & & & & & & & & & &\\
$21$ & $25$ & & & & & & & & & & & & & & & & & & & & & & & & & & &\\
$22$ & $27$ & & & & & & & & & & & & & & & & & & & & & & & & & & &\\
$23$ & $28$ & & & & & & & & & & & & & & & & & & & & & & & & & & &\\
$24$ & $29$ & & & & & & & & & & & & & & & & & & & & & & & & & & &\\
$25$ & $\ge 30$ & & & & & & & & & & & & & & & & & & & & & & & & & & &\\
$26$ & $\ge 31$ & & & & & & & & & & & & & & & & & & & & & & & & & & &\\
$27$ & $\ge 32$ & & & & & & & & & & & & & & & & & & & & & & & & & & &\\
\hline
\end{tabular}
}
\caption{Reflective Ramsey numbers $R_\mathrm{ref}(P^\mathrm{alt}_a, S^\mathrm{sc}_b)$ with $3 \le a \le 27$ and $3 \le b \le 30$.}
\label{palt_ssc_ref_tab}
\end{table}

\begin{table}[H]
\centering
\resizebox{\textwidth}{!}{
\begin{tabular}{|c||rrrrrrrrrrrrrrrrrrrrrrrrrrrr|}
\hline
\backslashbox{$a$}{$b$} & $3$ & $4$ & $5$ & $6$ & $7$ & $8$ & $9$ & $10$ & $11$ & $12$ & $13$ & $14$ & $15$ & $16$ & $17$ & $18$ & $19$ & $20$ & $21$ & $22$ & $23$ & $24$ & $25$ & $26$ & $27$ & $28$ & $29$ & $30$\\
\hline
\hline
$3$ & $3$ & $5$ & $5$ & $7$ & $7$ & $9$ & $9$ & $11$ & $11$ & $13$ & $13$ & $15$ & $15$ & $17$ & $17$ & $19$ & $\ge 19$ & $21$ & $\ge 21$ & $23$ & $\ge 23$ & & & & & & &\\
$4$ & $5$ & $6$ & $7$ & $8$ & $9$ & $10$ & $11$ & $12$ & $13$ & $14$ & $15$ & $16$ & $17$ & $18$ & $19$ & $20$ & $21$ & $22$ & $23$ & $24$ & $25$ & $26$ & $27$ & $28$ & $29$ & $30$ & $31$ & $32$\\
$5$ & $5$ & $7$ & $7$ & $9$ & $9$ & $11$ & $11$ & $13$ & $\ge 13$ & $\ge 15$ & $\ge 15$ & & & & & & & & & & & & & & & & &\\
$6$ & $7$ & $8$ & $9$ & $10$ & $11$ & $12$ & $\ge 13$ & $\ge 14$ & $\ge 15$ & & & & & & & & & & & & & & & & & & &\\
$7$ & $7$ & $9$ & $9$ & $11$ & $\ge 11$ & $\ge 13$ & $\ge 13$ & & & & & & & & & & & & & & & & & & & & &\\
$8$ & $9$ & $10$ & $11$ & $12$ & $\ge 13$ & & & & & & & & & & & & & & & & & & & & & & &\\
$9$ & $9$ & $11$ & $11$ & $\ge 13$ & $\ge 13$ & & & & & & & & & & & & & & & & & & & & & & &\\
$10$ & $11$ & $12$ & $13$ & $\ge 14$ & & & & & & & & & & & & & & & & & & & & & & & &\\
$11$ & $11$ & $13$ & $\ge 13$ & $\ge 15$ & & & & & & & & & & & & & & & & & & & & & & & &\\
$12$ & $13$ & $14$ & $\ge 15$ & & & & & & & & & & & & & & & & & & & & & & & & &\\
$13$ & $13$ & $15$ & $\ge 15$ & & & & & & & & & & & & & & & & & & & & & & & & &\\
$14$ & $15$ & $16$ & & & & & & & & & & & & & & & & & & & & & & & & & &\\
$15$ & $15$ & $\ge 17$ & & & & & & & & & & & & & & & & & & & & & & & & & &\\
$16$ & $17$ & $\ge 18$ & & & & & & & & & & & & & & & & & & & & & & & & & &\\
$17$ & $17$ & $\ge 19$ & & & & & & & & & & & & & & & & & & & & & & & & & &\\
$18$ & $19$ & & & & & & & & & & & & & & & & & & & & & & & & & & &\\
$19$ & $19$ & & & & & & & & & & & & & & & & & & & & & & & & & & &\\
$20$ & $21$ & & & & & & & & & & & & & & & & & & & & & & & & & & &\\
$21$ & $21$ & & & & & & & & & & & & & & & & & & & & & & & & & & &\\
$22$ & $23$ & & & & & & & & & & & & & & & & & & & & & & & & & & &\\
$23$ & $\ge 23$ & & & & & & & & & & & & & & & & & & & & & & & & & & &\\
$24$ & $25$ & & & & & & & & & & & & & & & & & & & & & & & & & & &\\
$25$ & $\ge 25$ & & & & & & & & & & & & & & & & & & & & & & & & & & &\\
$26$ & $\ge 27$ & & & & & & & & & & & & & & & & & & & & & & & & & & &\\
\hline
\end{tabular}
}
\caption{Dihedral Ramsey numbers $R_\mathrm{dih}(P^\mathrm{alt}_a, S^\mathrm{sc}_b)$ with $3 \le a \le 26$ and $3 \le b \le 30$.}
\label{palt_ssc_dih_tab}
\end{table}

\begin{table}[H]
\centering
\resizebox{0.7\textwidth}{!}{
\begin{tabular}{|c||rrrrrrrrrrrrrr|}
\hline
\backslashbox{$a$}{$b$} & $3$ & $4$ & $5$ & $6$ & $7$ & $8$ & $9$ & $10$ & $11$ & $12$ & $13$ & $14$ & $15$ & $16$\\
\hline
\hline
$3$ & $3$ & $5$ & $7$ & $9$ & $11$ & $13$ & $15$ & $17$ & $19$ & $21$ & $23$ & $\ge 24$ & $\ge 25$ & $\ge 25$\\
$4$ & $5$ & $7$ & $10$ & $13$ & $16$ & $19$ & $22$ & $\ge 25$ & $\ge 25$ & $\ge 25$ & & & &\\
$5$ & $5$ & $9$ & $13$ & $17$ & $21$ & $25$ & $\ge 26$ & $\ge 26$ & $\ge 26$ & & & & &\\
$6$ & $7$ & $11$ & $16$ & $21$ & $26$ & $\ge 29$ & $\ge 29$ & $\ge 29$ & & & & & &\\
$7$ & $7$ & $13$ & $19$ & $25$ & $\ge 31$ & $\ge 31$ & $\ge 31$ & & & & & & &\\
$8$ & $9$ & $15$ & $22$ & $\ge 29$ & $\ge 31$ & & & & & & & & &\\
$9$ & $9$ & $17$ & $\ge 25$ & $\ge 29$ & $\ge 31$ & & & & & & & & &\\
$10$ & $11$ & $19$ & $\ge 25$ & $\ge 29$ & & & & & & & & & &\\
$11$ & $11$ & $\ge 21$ & $\ge 25$ & & & & & & & & & & &\\
$12$ & $13$ & $\ge 23$ & & & & & & & & & & & &\\
$13$ & $13$ & $\ge 24$ & & & & & & & & & & & &\\
$14$ & $15$ & & & & & & & & & & & & &\\
$15$ & $15$ & & & & & & & & & & & & &\\
$16$ & $17$ & & & & & & & & & & & & &\\
$17$ & $17$ & & & & & & & & & & & & &\\
$18$ & $19$ & & & & & & & & & & & & &\\
$19$ & $19$ & & & & & & & & & & & & &\\
$20$ & $21$ & & & & & & & & & & & & &\\
$21$ & $21$ & & & & & & & & & & & & &\\
$22$ & $23$ & & & & & & & & & & & & &\\
$23$ & $\ge 23$ & & & & & & & & & & & & &\\
$24$ & $25$ & & & & & & & & & & & & &\\
$25$ & $\ge 25$ & & & & & & & & & & & & &\\
$26$ & $\ge 27$ & & & & & & & & & & & & &\\
\hline
\end{tabular}
}
\caption{Dihedral Ramsey numbers $R_\mathrm{dih}(P^\mathrm{alt}_a, P^\mathrm{mon}_b)$ with $3 \le a \le 26$ and $3 \le b \le 16$.}
\label{palt_pmon_dih_tab}
\end{table}

The computational results for the reflective Ramsey numbers of alternating paths versus monotone cycles suggest that, provided $a, b \ge 3$, we have $R_\mathrm{ref}(P^\mathrm{alt}_a, C^\mathrm{mon}_b) = R_\mathrm{ord}(P^\mathrm{alt}_a, C^\mathrm{mon}_b)$ if $a$ is odd, and $R_\mathrm{ref}(P^\mathrm{alt}_a, C^\mathrm{mon}_b) = R_\mathrm{ord}(P^\mathrm{alt}_a, C^\mathrm{mon}_b) - 1$ if $a$ is even; see Table \ref{palt_cmon_ref_tab}. This points to a close relationship between the two types of Ramsey numbers. As for the dihedral Ramsey numbers, they once again appear to coincide with the cyclic Ramsey numbers, as shown in Table \ref{palt_cmon_dih_tab}. We therefore pose the following two conjectures.

\begin{conjecture}\label{palt_cmon_ref_conj}
    For any $a, b \ge 3$, we have $R_\mathrm{ref}(P^\mathrm{alt}_a, C^\mathrm{mon}_b) = \lfloor (a - 1)(b - \frac{1}{2}) \rfloor$.
\end{conjecture}

\begin{conjecture}\label{palt_cmon_dih_conj}
    For any $a \in \mathbb{N}$ and $b \ge 2$, we have $R_\mathrm{dih}(P^\mathrm{alt}_a, C^\mathrm{mon}_b) = 1 + (a - 1)(b - 1)$.
\end{conjecture}

\begin{table}[H]
\centering
\footnotesize
\begin{tabular}{|c||rrrrrrrrrrrrrr|}
\hline
\backslashbox{$a$}{$b$} & $3$ & $4$ & $5$ & $6$ & $7$ & $8$ & $9$ & $10$ & $11$ & $12$ & $13$ & $14$ & $15$ & $16$\\
\hline
\hline
$3$ & $5$ & $7$ & $9$ & $11$ & $13$ & $15$ & $17$ & $19$ & $21$ & $23$ & $25$ & $\ge 27$ & $\ge 28$ & $\ge 28$\\
$4$ & $7$ & $10$ & $13$ & $16$ & $19$ & $22$ & $\ge 25$ & $\ge 28$ & $\ge 29$ & & & & &\\
$5$ & $10$ & $14$ & $18$ & $22$ & $\ge 26$ & $\ge 30$ & $\ge 33$ & & & & & & &\\
$6$ & $12$ & $17$ & $22$ & $\ge 27$ & $\ge 32$ & $\ge 36$ & & & & & & & &\\
$7$ & $15$ & $\ge 21$ & $\ge 27$ & $\ge 33$ & & & & & & & & & &\\
$8$ & $17$ & $\ge 24$ & $\ge 31$ & $\ge 36$ & & & & & & & & & &\\
$9$ & $\ge 20$ & $\ge 27$ & $\ge 33$ & & & & & & & & & & &\\
$10$ & $\ge 22$ & & & & & & & & & & & & &\\
$11$ & $\ge 25$ & & & & & & & & & & & & &\\
\hline
\end{tabular}
\caption{Reflective Ramsey numbers $R_\mathrm{ref}(P^\mathrm{alt}_a, C^\mathrm{mon}_b)$ with $3 \le a \le 11$ and $3 \le b \le 16$.}
\label{palt_cmon_ref_tab}
\end{table}

\begin{table}[H]
\centering
\footnotesize
\begin{tabular}{|c||rrrrrrrrrrrrr|}
\hline
\backslashbox{$a$}{$b$} & $3$ & $4$ & $5$ & $6$ & $7$ & $8$ & $9$ & $10$ & $11$ & $12$ & $13$ & $14$ & $15$\\
\hline
\hline
$3$ & $5$ & $7$ & $9$ & $11$ & $13$ & $15$ & $17$ & $19$ & $21$ & $23$ & $\ge 24$ & $\ge 24$ & $\ge 24$\\
$4$ & $7$ & $10$ & $13$ & $16$ & $19$ & $22$ & $\ge 25$ & $\ge 25$ & $\ge 24$ & & & &\\
$5$ & $9$ & $13$ & $17$ & $21$ & $\ge 25$ & $\ge 27$ & $\ge 27$ & & & & & &\\
$6$ & $11$ & $16$ & $21$ & $\ge 26$ & $\ge 29$ & $\ge 29$ & & & & & & &\\
$7$ & $13$ & $19$ & $\ge 25$ & $\ge 29$ & & & & & & & & &\\
$8$ & $15$ & $\ge 22$ & $\ge 26$ & $\ge 30$ & & & & & & & & &\\
$9$ & $\ge 17$ & $\ge 24$ & $\ge 27$ & & & & & & & & & &\\
$10$ & $\ge 19$ & $\ge 24$ & & & & & & & & & & &\\
$11$ & $\ge 21$ & & & & & & & & & & & &\\
\hline
\end{tabular}
\caption{Dihedral Ramsey numbers $R_\mathrm{dih}(P^\mathrm{alt}_a, C^\mathrm{mon}_b)$ with $3 \le a \le 11$ and $3 \le b \le 15$.}
\label{palt_cmon_dih_tab}
\end{table}

The computational results for the reflective and dihedral Ramsey numbers of alternating paths versus complete graphs are given in Tables \ref{palt_k_ref_tab} and \ref{palt_k_dih_tab}. Although the reflective Ramsey numbers seem to have values similar to the corresponding ordered Ramsey numbers, the underlying pattern is difficult to discern. The dihedral Ramsey numbers, on the other hand, appear to coincide with the cyclic Ramsey numbers, leading to the following conjecture.

\begin{conjecture}\label{palt_k_dih_conj}
    For any $a, b \in \mathbb{N}$, we have $R_\mathrm{dih}(P^\mathrm{alt}_a, K_b) = 1 + (a - 1)(b - 1)$.
\end{conjecture}

We conclude the investigation of Ramsey numbers involving alternating paths with Tables \ref{palt_mnest_ref_tab} and \ref{palt_mnest_dih_tab}, which contain the computational results for the reflective and dihedral Ramsey numbers of alternating paths versus nested matchings. As in the case of ordered and cyclic Ramsey numbers, the reflective and dihedral Ramsey numbers for this combination of graph classes also exhibit a structure that seems difficult to discern.

\begin{table}[H]
\centering
\footnotesize
\begin{tabular}{|c||rrrrrrrrr|}
\hline
\backslashbox{$a$}{$b$} & $3$ & $4$ & $5$ & $6$ & $7$ & $8$ & $9$ & $10$ & $11$\\
\hline
\hline
$3$ & $5$ & $7$ & $9$ & $11$ & $13$ & $15$ & $\ge 17$ & $\ge 19$ & $\ge 21$\\
$4$ & $7$ & $11$ & $14$ & $\ge 17$ & $\ge 21$ & $\ge 24$ & & &\\
$5$ & $10$ & $14$ & $\ge 19$ & $\ge 23$ & $\ge 28$ & & & &\\
$6$ & $12$ & $\ge 18$ & $\ge 24$ & & & & & &\\
$7$ & $15$ & $\ge 22$ & $\ge 28$ & & & & & &\\
$8$ & $17$ & $\ge 25$ & & & & & & &\\
$9$ & $\ge 20$ & & & & & & & &\\
$10$ & $\ge 22$ & & & & & & & &\\
$11$ & $\ge 25$ & & & & & & & &\\
\hline
\end{tabular}
\caption{Reflective Ramsey numbers $R_\mathrm{ref}(P^\mathrm{alt}_a, K_b)$ with $3 \le a \le 11$ and $3 \le b \le 11$.}
\label{palt_k_ref_tab}
\end{table}

\begin{table}[H]
\centering
\resizebox{0.85\textwidth}{!}{
\begin{tabular}{|c||rrrrrrrrrrrrrr|}
\hline
\backslashbox{$a$}{$b$} & $4$ & $6$ & $8$ & $10$ & $12$ & $14$ & $16$ & $18$ & $20$ & $22$ & $24$ & $26$ & $28$ & $30$\\
\hline
\hline
$3$ & $6$ & $8$ & $11$ & $13$ & $15$ & $18$ & $20$ & $22$ & $24$ & $27$ & $29$ & $31$ & $33$ & $35$\\
$4$ & $7$ & $10$ & $12$ & $15$ & $17$ & $20$ & $22$ & $24$ & $27$ & $\ge 29$ & $\ge 31$ & $\ge 33$ & &\\
$5$ & $8$ & $11$ & $14$ & $17$ & $19$ & $21$ & $24$ & $\ge 26$ & $\ge 28$ & $\ge 31$ & & & &\\
$6$ & $9$ & $12$ & $15$ & $18$ & $21$ & $23$ & $\ge 26$ & $\ge 28$ & $\ge 30$ & & & & &\\
$7$ & $10$ & $13$ & $16$ & $19$ & $22$ & $\ge 25$ & $\ge 27$ & $\ge 29$ & & & & & &\\
$8$ & $11$ & $14$ & $17$ & $20$ & $23$ & $\ge 26$ & $\ge 28$ & $\ge 29$ & & & & & &\\
$9$ & $12$ & $15$ & $18$ & $21$ & $\ge 24$ & $\ge 26$ & & & & & & & &\\
$10$ & $13$ & $16$ & $19$ & $22$ & $\ge 25$ & $\ge 27$ & & & & & & & &\\
$11$ & $14$ & $17$ & $20$ & $23$ & $\ge 26$ & & & & & & & & &\\
$12$ & $15$ & $18$ & $21$ & $\ge 24$ & & & & & & & & & &\\
$13$ & $16$ & $19$ & $22$ & $\ge 25$ & & & & & & & & & &\\
$14$ & $17$ & $20$ & $23$ & $\ge 26$ & & & & & & & & & &\\
$15$ & $18$ & $21$ & $24$ & & & & & & & & & & &\\
$16$ & $19$ & $22$ & $25$ & & & & & & & & & & &\\
$17$ & $20$ & $23$ & $\ge 26$ & & & & & & & & & & &\\
$18$ & $21$ & $24$ & $\ge 27$ & & & & & & & & & & &\\
$19$ & $22$ & $25$ & $\ge 28$ & & & & & & & & & & &\\
$20$ & $23$ & $26$ & & & & & & & & & & & &\\
$21$ & $24$ & $27$ & & & & & & & & & & & &\\
$22$ & $25$ & $28$ & & & & & & & & & & & &\\
$23$ & $26$ & $29$ & & & & & & & & & & & &\\
$24$ & $27$ & $30$ & & & & & & & & & & & &\\
$25$ & $28$ & $31$ & & & & & & & & & & & &\\
$26$ & $29$ & $32$ & & & & & & & & & & & &\\
$27$ & $30$ & $33$ & & & & & & & & & & & &\\
$28$ & $31$ & $34$ & & & & & & & & & & & &\\
$29$ & $32$ & $35$ & & & & & & & & & & & &\\
$30$ & $33$ & $36$ & & & & & & & & & & & &\\
\hline
\end{tabular}
}
\caption{Reflective Ramsey numbers $R_\mathrm{ref}(P^\mathrm{alt}_a, M^\mathrm{nest}_b)$ with any $3 \le a \le 30$ and even $4 \le b \le 30$.}
\label{palt_mnest_ref_tab}
\end{table}

\begin{table}[H]
\centering
\resizebox{0.85\textwidth}{!}{
\begin{tabular}{|c||rrrrrrrrrrrrrr|}
\hline
\backslashbox{$a$}{$b$} & $4$ & $6$ & $8$ & $10$ & $12$ & $14$ & $16$ & $18$ & $20$ & $22$ & $24$ & $26$ & $28$ & $30$\\
\hline
\hline
$3$ & $4$ & $7$ & $8$ & $11$ & $12$ & $15$ & $16$ & $19$ & $20$ & $23$ & $24$ & $27$ & $\ge 28$ & $31$\\
$4$ & $6$ & $8$ & $10$ & $12$ & $14$ & $16$ & $18$ & $20$ & $\ge 22$ & $\ge 24$ & $\ge 26$ & & &\\
$5$ & $6$ & $9$ & $10$ & $13$ & $15$ & $17$ & $\ge 19$ & $\ge 21$ & $\ge 23$ & & & & &\\
$6$ & $8$ & $10$ & $12$ & $14$ & $16$ & $\ge 18$ & $\ge 20$ & $\ge 22$ & & & & & &\\
$7$ & $8$ & $11$ & $12$ & $15$ & $\ge 17$ & $\ge 19$ & $\ge 21$ & & & & & & &\\
$8$ & $10$ & $12$ & $14$ & $16$ & $\ge 18$ & & & & & & & & &\\
$9$ & $10$ & $13$ & $14$ & $17$ & $\ge 18$ & & & & & & & & &\\
$10$ & $12$ & $14$ & $16$ & $\ge 18$ & & & & & & & & & &\\
$11$ & $12$ & $15$ & $16$ & $\ge 19$ & & & & & & & & & &\\
$12$ & $14$ & $16$ & $18$ & $\ge 20$ & & & & & & & & & &\\
$13$ & $14$ & $17$ & $\ge 18$ & & & & & & & & & & &\\
$14$ & $16$ & $18$ & $20$ & & & & & & & & & & &\\
$15$ & $16$ & $19$ & $\ge 20$ & & & & & & & & & & &\\
$16$ & $18$ & $20$ & $22$ & & & & & & & & & & &\\
$17$ & $18$ & $21$ & $\ge 22$ & & & & & & & & & & &\\
$18$ & $20$ & $22$ & & & & & & & & & & & &\\
$19$ & $20$ & $23$ & & & & & & & & & & & &\\
$20$ & $22$ & $24$ & & & & & & & & & & & &\\
$21$ & $22$ & $25$ & & & & & & & & & & & &\\
$22$ & $24$ & $26$ & & & & & & & & & & & &\\
$23$ & $24$ & $27$ & & & & & & & & & & & &\\
$24$ & $26$ & $28$ & & & & & & & & & & & &\\
$25$ & $26$ & $29$ & & & & & & & & & & & &\\
$26$ & $28$ & $30$ & & & & & & & & & & & &\\
$27$ & $28$ & $31$ & & & & & & & & & & & &\\
$28$ & $30$ & $32$ & & & & & & & & & & & &\\
$29$ & $30$ & $33$ & & & & & & & & & & & &\\
$30$ & $32$ & $34$ & & & & & & & & & & & &\\
\hline
\end{tabular}
}
\caption{Dihedral Ramsey numbers $R_\mathrm{dih}(P^\mathrm{alt}_a, M^\mathrm{nest}_b)$ with any $3 \le a \le 30$ and even $4 \le b \le 30$.}
\label{palt_mnest_dih_tab}
\end{table}

\begin{table}[H]
\centering
\footnotesize
\begin{tabular}{|c||rrrrrrrrr|}
\hline
\backslashbox{$a$}{$b$} & $3$ & $4$ & $5$ & $6$ & $7$ & $8$ & $9$ & $10$ & $11$\\
\hline
\hline
$3$ & $5$ & $7$ & $9$ & $11$ & $13$ & $15$ & $\ge 17$ & $\ge 19$ & $\ge 21$\\
$4$ & $7$ & $10$ & $13$ & $\ge 16$ & $\ge 19$ & $\ge 22$ & & &\\
$5$ & $9$ & $13$ & $\ge 17$ & $\ge 21$ & $\ge 25$ & & & &\\
$6$ & $11$ & $16$ & $\ge 21$ & & & & & &\\
$7$ & $13$ & $\ge 19$ & $\ge 25$ & & & & & &\\
$8$ & $15$ & $\ge 22$ & & & & & & &\\
$9$ & $\ge 17$ & $\ge 25$ & & & & & & &\\
$10$ & $\ge 19$ & & & & & & & &\\
$11$ & $\ge 21$ & & & & & & & &\\
\hline
\end{tabular}
\caption{Dihedral Ramsey numbers $R_\mathrm{dih}(P^\mathrm{alt}_a, K_b)$ with $3 \le a \le 11$ and $3 \le b \le 11$.}
\label{palt_k_dih_tab}
\end{table}

The ordered and reflective Ramsey numbers of start-central stars versus start-central stars are straightforward to determine, as shown in the following proposition, which directly extends an earlier result by Balko, Cibulka, Král and Kynčl \cite[Observation~12]{BalCiKralKyn2020}.

\begin{proposition}\label{ssc_ssc_ref_prop}
    For any $k \in \mathbb{N}$ and $n_1, n_2, \ldots, n_k \ge 2$, we have
    \[
        R_\mathrm{ord}(S^\mathrm{sc}_{n_1}, S^\mathrm{sc}_{n_2}, \ldots, S^\mathrm{sc}_{n_k}) = R_\mathrm{ref}(S^\mathrm{sc}_{n_1}, S^\mathrm{sc}_{n_2}, \ldots, S^\mathrm{sc}_{n_k}) = \sum_{j = 1}^k (n_j - 2) + 2 .
    \]
\end{proposition}
\begin{proof}
    Let $\Gamma_1, \Gamma_2, \ldots, \Gamma_k$ be the permutation groups associated with $S^\mathrm{sc}_{n_1}, S^\mathrm{sc}_{n_2}, \ldots, S^\mathrm{sc}_{n_k}$, respectively, in the considered permutational Ramsey number $R_\mathrm{type}(S^\mathrm{sc}_{n_1}, S^\mathrm{sc}_{n_2}, \ldots, S^\mathrm{sc}_{n_k})$ with $\mathrm{type} \in \{ \mathrm{ord}, \mathrm{ref} \}$. First, we establish the upper bound. Let $n = \sum_{j = 1}^k (n_j - 2) + 2$, and consider a $k$-edge-coloring of $K_n$. Let $G_j$ denote the spanning subgraph of $K_n$ comprising the edges of color $j$, for $j \in \{ 1, 2, \ldots, k \}$. By the pigeonhole principle, vertex $0$ is incident to at least $n_j - 1$ edges of color $j$, for some $j \in \{ 1, 2, \ldots, k \}$. Thus, there exists an increasing injective homomorphism from $S^\mathrm{sc}_{n_j}$ to $G_j$. Therefore, $S^\mathrm{sc}_{n_j}$ is $\Gamma_j$-embeddable in $G_j$, which implies that $R_\mathrm{type}(S^\mathrm{sc}_{n_1}, S^\mathrm{sc}_{n_2}, \ldots, S^\mathrm{sc}_{n_k}) \le \sum_{j = 1}^k (n_j - 2) + 2$.

    Now, we establish the lower bound. Let $n = \sum_{j = 1}^k (n_j - 2) + 1$, and consider the $k$-edge-coloring of $K_n$ in which the edge between any two distinct vertices $u, v \in V(K_n)$ is colored with the smallest $\ell \in \{ 1, 2, \ldots, k \}$ such that $|v - u| \le \sum_{j = 1}^{\ell} (n_j - 2)$. Let $G_1, G_2, \ldots, G_k$ be defined as above, and by way of contradiction, suppose that, for some $j \in \{ 1, 2, \ldots, k \}$, $S^\mathrm{sc}_{n_j}$ is $\Gamma_j$-embeddable in $G_j$. Then there exists a strictly monotone injective homomorphism $\varphi \colon S^\mathrm{sc}_{n_j} \to G_j$. Therefore, in $G_j$, vertex $\varphi(0)$ has at least $n_j - 1$ smaller neighbors or at least $n_j - 1$ larger neighbors, which is impossible. We conclude that $R_\mathrm{type}(S^\mathrm{sc}_{n_1}, S^\mathrm{sc}_{n_2}, \ldots, S^\mathrm{sc}_{n_k}) \ge \sum_{j = 1}^k (n_j - 2) + 2$.
\end{proof}

\noindent
We note that the cyclic Ramsey numbers of start-central stars versus start-central stars coincide with the corresponding standard Ramsey numbers, which were determined by Burr and Roberts \cite{BurrRo1973}.

The reflective Ramsey numbers of start-central stars versus monotone paths are fully determined by Corollary~\ref{ssc_pmon_ref_cor}, and they coincide with the corresponding ordered Ramsey numbers; see Table \ref{summary_tab}. As for start-central stars versus monotone cycles, the ordered and cyclic Ramsey numbers were recently determined \cite{BaDamSteSto2026}, as follows.

\begin{theorem}[\hspace{1sp}{\cite[Theorem~4.18 and Corollary~4.19]{BaDamSteSto2026}}]\label{ssc_cmon_ord_cyc_th}
    For any $a \in \mathbb{N}$ and $b \ge 2$, we have $R_\mathrm{ord}(S_a^\mathrm{sc}, C_b^\mathrm{mon}) = R_\mathrm{cyc}(S_a^\mathrm{sc}, C_b^\mathrm{mon}) = 1 + (a - 1)(b - 1)$.
\end{theorem}

\noindent
The reflective Ramsey numbers can now easily be determined as shown in the following proposition.

\begin{proposition}\label{ssc_cmon_ref_prop}
    For any $a \in \mathbb{N}$ and $b \ge 2$, we have $R_\mathrm{ref}(S_a^\mathrm{sc}, C_b^\mathrm{mon}) = 1 + (a - 1)(b - 1)$.
\end{proposition}
\begin{proof}
    Since $P^\mathrm{mon}_b$ is a subgraph of $C^\mathrm{mon}_b$, Corollary \ref{ssc_pmon_ref_cor} yields
    \[
        R_\mathrm{ref}(S_a^\mathrm{sc}, C_b^\mathrm{mon}) \ge R_\mathrm{ref}(S_a^\mathrm{sc}, P_b^\mathrm{mon}) =  1 + (a - 1)(b - 1) .
    \]
    By Corollary \ref{sandwich_cor} and Theorem \ref{ssc_cmon_ord_cyc_th}, we obtain
    \[
        R_\mathrm{ref}(S_a^\mathrm{sc}, C_b^\mathrm{mon}) \le R_\mathrm{ord}(S_a^\mathrm{sc}, C_b^\mathrm{mon}) =  1 + (a - 1)(b - 1) .
    \]
    Therefore, $R_\mathrm{ref}(S_a^\mathrm{sc}, C_b^\mathrm{mon}) = 1 + (a - 1)(b - 1)$.
\end{proof}

\noindent
Hence, the ordered, reflective, cyclic and dihedral Ramsey numbers all coincide when one of the arguments is a start-central star, while the other is a monotone cycle.

We now turn to the reflective Ramsey numbers of start-central stars versus complete graphs. By Proposition~\ref{equiv_class}, the cyclic and dihedral Ramsey numbers for this combination of graphs coincide with the standard Ramsey numbers. We recall the following well-known result of Chvátal \cite{Chvatal1977}.

\begin{theorem}[\hspace{1sp}{\cite{Chvatal1977}}]\label{chvatal_cool_th}
    For any $a, b \in \mathbb{N}$ and tree $T$ of order $a$, we have $R(T, K_b) = 1 + (a - 1)(b - 1)$.
\end{theorem}

\noindent
As an immediate consequence of Theorem \ref{chvatal_cool_th}, we obtain the following corollary.

\begin{corollary}\label{ssc_k_cyc_dih_cor}
    For any $a, b \in \mathbb{N}$, we have $R_\mathrm{cyc}(S^\mathrm{sc}_a, K_b) = R_\mathrm{dih}(S^\mathrm{sc}_a, K_b) = 1 + (a - 1)(b - 1)$.
\end{corollary}

\noindent
The ordered Ramsey numbers of start-central stars versus complete graphs were also recently computed \cite{BaDamSteSto2026}.

\begin{theorem}[\hspace{1sp}{\cite[Theorem~4.24]{BaDamSteSto2026}}]\label{ssc_k_ord_th}
    For any $a, b \in \mathbb{N}$, we have $R_\mathrm{ord}(S^\mathrm{sc}_a, K_b) = 1 + (a - 1)(b - 1)$.
\end{theorem}

\noindent
We now obtain the following proposition from Theorem \ref{ssc_k_ord_th}, Corollary \ref{ssc_k_cyc_dih_cor} and Proposition \ref{sandwich_prop}.

\begin{corollary}\label{ssc_k_ref_cor}
    For any $a, b \in \mathbb{N}$, we have $R_\mathrm{ref}(S^\mathrm{sc}_a, K_b) = 1 + (a - 1)(b - 1)$.
\end{corollary}

\noindent
Therefore, the ordered, reflective, cyclic and dihedral Ramsey numbers all coincide for start-central stars versus complete graphs.

We conclude this section by considering the reflective Ramsey numbers of start-central stars versus nested matchings. In this case, the reflective Ramsey numbers appear to behave in a complicated manner; see Table~\ref{ssc_mnest_ref_tab}. Nonetheless, these numbers seem to be close to the corresponding ordered Ramsey numbers.

\begin{table}[H]
\centering
\resizebox{0.83\textwidth}{!}{
\begin{tabular}{|c||rrrrrrrrrrrrrr|}
\hline
\backslashbox{$a$}{$b$} & $4$ & $6$ & $8$ & $10$ & $12$ & $14$ & $16$ & $18$ & $20$ & $22$ & $24$ & $26$ & $28$ & $30$\\
\hline
\hline
$3$ & $6$ & $8$ & $11$ & $13$ & $15$ & $18$ & $20$ & $22$ & $24$ & $27$ & $29$ & $31$ & $33$ & $35$\\
$4$ & $7$ & $10$ & $12$ & $15$ & $18$ & $20$ & $22$ & $25$ & $\ge 27$ & $\ge 29$ & $\ge 31$ & & &\\
$5$ & $8$ & $11$ & $14$ & $17$ & $20$ & $22$ & $\ge 24$ & $\ge 27$ & $\ge 29$ & & & & &\\
$6$ & $10$ & $13$ & $16$ & $19$ & $21$ & $\ge 24$ & $\ge 26$ & $\ge 28$ & & & & & &\\
$7$ & $11$ & $14$ & $17$ & $20$ & $\ge 23$ & $\ge 25$ & $\ge 28$ & & & & & & &\\
$8$ & $12$ & $15$ & $19$ & $22$ & $\ge 24$ & & & & & & & & &\\
$9$ & $13$ & $17$ & $20$ & $\ge 23$ & $\ge 26$ & & & & & & & & &\\
$10$ & $15$ & $18$ & $22$ & $\ge 25$ & & & & & & & & & &\\
$11$ & $16$ & $20$ & $23$ & $\ge 26$ & & & & & & & & & &\\
$12$ & $17$ & $21$ & $\ge 24$ & & & & & & & & & & &\\
$13$ & $18$ & $22$ & $\ge 25$ & & & & & & & & & & &\\
$14$ & $19$ & $23$ & $\ge 27$ & & & & & & & & & & &\\
$15$ & $21$ & $25$ & & & & & & & & & & & &\\
$16$ & $22$ & $26$ & & & & & & & & & & & &\\
$17$ & $23$ & $27$ & & & & & & & & & & & &\\
$18$ & $24$ & $\ge 28$ & & & & & & & & & & & &\\
$19$ & $25$ & $\ge 29$ & & & & & & & & & & & &\\
$20$ & $26$ & $\ge 30$ & & & & & & & & & & & &\\
$21$ & $28$ & & & & & & & & & & & & &\\
$22$ & $29$ & & & & & & & & & & & & &\\
$23$ & $30$ & & & & & & & & & & & & &\\
$24$ & $31$ & & & & & & & & & & & & &\\
$25$ & $32$ & & & & & & & & & & & & &\\
$26$ & $33$ & & & & & & & & & & & & &\\
$27$ & $34$ & & & & & & & & & & & & &\\
$28$ & $\ge 36$ & & & & & & & & & & & & &\\
$29$ & $\ge 37$ & & & & & & & & & & & & &\\
$30$ & $\ge 38$ & & & & & & & & & & & & &\\
\hline
\end{tabular}
}
\caption{Reflective Ramsey numbers $R_\mathrm{ref}(S^\mathrm{sc}_a, M^\mathrm{nest}_b)$ with any $3 \le a \le 30$ and even $4 \le b \le 30$.}
\label{ssc_mnest_ref_tab}
\end{table}

\section{Conclusion}\label{sc:conc}

We have shown how the SAT-based approach originally proposed by Poljak \cite{Poljak2020} for ordered Ramsey numbers and recently extended to cyclic Ramsey numbers \cite{BaDamSteSto2026} can be generalized to cover permutational Ramsey numbers with arbitrary associated permutation groups. Using this approach, we have computed small two-color reflective and dihedral Ramsey numbers for the following graph classes: monotone and alternating paths, monotone cycles, start-central stars, complete graphs and nested matchings. The extensive computation was automated through a pipeline of four \texttt{Python} modules based on the source code from \cite{BaDamSteSto2026Repo}, with significant optimizations made to improve the API and remove redundancy. We have also established several general results and formulated conjectures motivated by the computational findings.

As it turns out, although capable of obtaining lower bounds for small Ramsey numbers of both reflective and dihedral types, the SAT-based approach appears to be more efficient at providing upper bounds for reflective Ramsey numbers than for dihedral Ramsey numbers. This can be explained by the fact that the generated SAT problem has a number of clauses that grows linearly with the order of each associated permutation group. On the other hand, by
Proposition~\ref{sandwich_prop}, larger associated permutation
groups generally lead to smaller permutational Ramsey numbers.
Consequently, increasing the order of the associated permutation groups has two competing effects: while the generated SAT instances become more difficult to solve for fixed $n \in \mathbb{N}$, the Ramsey numbers themselves tend to decrease, reducing the values of $n \in \mathbb{N}$ that need to be considered. This observation is also supported by the experimental findings of \cite{BaDamSteSto2026}, which show that for fixed $n \in \mathbb{N}$, SAT problems corresponding to ordered Ramsey numbers are generally easier to solve than those corresponding to cyclic Ramsey numbers. Our computational results further indicate that the reflective Ramsey numbers are often very close to the corresponding ordered Ramsey numbers, while the dihedral Ramsey numbers frequently coincide with the cyclic Ramsey numbers when one of the arguments is an alternating path.

We conclude with the remark that it may be worthwhile to investigate more general permutational Ramsey numbers in which the permutation groups associated with graphs are tailored to the corresponding graphs. For instance, let $n \in \mathbb{N}$, and consider the \emph{monotone fan} $F^\mathrm{mon}_n$, defined as the fan graph in which $0$ is adjacent to all the other vertices, while any two nonzero vertices are adjacent if and only if they are consecutive integers. Equivalently, $F^\mathrm{mon}_n$ is the graph of order $n$ with edge set $E(S^\mathrm{sc}_n) \cup E(P^\mathrm{mon}_n)$. Since $0$ is a universal vertex, it is natural to associate this graph with a permutation group that fixes $0$. For example, in analogy with the cyclic shift permutation, one could consider the permutation group $\left\langle \begin{pmatrix}\begin{smallmatrix}
    0 & 1 & 2 & \cdots & n - 2 & n - 1\\
    0 & 2 & 3 & \cdots & n - 1 & 1
\end{smallmatrix}\end{pmatrix} \right\rangle$. This example illustrates one possible way in which the associated permutation groups can be chosen according to the structure of the underlying graphs. More generally, this motivates the study of permutational Ramsey numbers with arbitrary choices of associated permutation groups. Theorem~\ref{brutal_th} illustrates that allowing arbitrary associated permutation groups naturally leads to general results in certain cases. We believe that this could provide a fruitful direction for future research.

\section*{Acknowledgments}

I.\ Damnjanović is supported by the Ministry of Science, Technological Development and Innovation of the Republic of Serbia, grant number 451-03-34/2026-03/200102, and the Science Fund of the Republic of Serbia, grant \#6767, Lazy walk counts and spectral radius of threshold graphs --- LZWK.

\section*{Conflict of interest}

The authors declare that they have no conflict of interest.

\end{document}